\documentclass[12pt, leqno]{article}
\usepackage{amsmath,amsfonts,amsthm,amscd,amssymb,verbatim} 
\usepackage[all]{xy}

\numberwithin{equation}{section}

\theoremstyle{plain}
\newtheorem{thm}{Theorem}[section]

\newtheorem{prop}[thm]{Proposition}
\newtheorem{lem}[thm]{Lemma}
\newtheorem{cor}[thm]{Corollary}

\theoremstyle{definition}
\newtheorem{rem}[thm]{Remark}

\newcommand{\beast}{\begin{eqnarray*}}
\newcommand{\east}{\end{eqnarray*}}

\newcommand{\N}{{\Bbb N}}
\newcommand{\Z}{{\Bbb Z}}
\renewcommand{\P}{{\Bbb P}}
\newcommand{\fP}{\wh{{\Bbb P}}}

\newcommand{\R}{{\Bbb R}}

\newcommand{\D}{{\bold D}}
\newcommand{\G}{{\Bbb G}}
\newcommand{\fG}{\wh{{\Bbb G}}}
\newcommand{\E}{{\bold E}}

\newcommand{\X}{{\bold X}}
\newcommand{\Y}{{\bold Y}}
\newcommand{\Af}{{\Bbb A}}
\newcommand{\fAf}{\wh{{\Bbb A}}}

\newcommand{\Spec}{{\mathrm{Spec}}\,}

\newcommand{\Spf}{{\mathrm{Spf}}\,}

\newcommand{\lra}{\longrightarrow}

\newcommand{\hra}{\hookrightarrow}

\newcommand{\lla}{\longleftarrow}

\newcommand{\ti}[1]{\widetilde{#1}}

\newcommand{\ul}[1]{\underline{#1}}
\newcommand{\ol}[1]{\overline{#1}}
\newcommand{\os}{\overset}

\newcommand{\Hom}{{\mathrm{Hom}}}

\newcommand{\an}{{\mathrm{an}}}

\newcommand{\Frac}{{\mathrm{Frac}}\,}

\newcommand{\cC}{{\cal C}}
\newcommand{\cD}{{\cal D}}
\newcommand{\cE}{{\cal E}}

\newcommand{\cH}{{\cal H}}

\newcommand{\cL}{{\cal L}}
\newcommand{\cO}{{\cal O}}
\newcommand{\cP}{{\cal P}}
\newcommand{\cQ}{{\cal Q}}

\newcommand{\cU}{{\cal U}}
\newcommand{\cV}{{\cal V}}
\newcommand{\cX}{{\cal X}}
\newcommand{\cY}{{\cal Y}}
\newcommand{\cZ}{{\cal Z}}

\newcommand{\fU}{{\frak U}}
\newcommand{\fV}{{\mathfrak{V}}}

\newcommand{\rk}{{\mathrm{rk}}\,}

\renewcommand{\sp}{{\mathrm{sp}}}

\newcommand{\wh}{\widehat}

\newcommand{\tr}{{\mathrm{tr}}}

\renewcommand{\d}{\dagger}

\newcommand{\lam}{\lambda}

\newcommand{\pa}{\partial}
\newcommand{\Mat}{{\mathrm{Mat}}}
\newcommand{\e}{{\bold e}}
\newcommand{\f}{{\bold f}}

\renewcommand{\1}{{\bold 1}}
\newcommand{\rrho}{{\boldsymbol \rho}}

\newcommand{\ad}{{\mathrm{ad}}}

\newcommand{\MIC}{{\mathrm{MIC}}}
\newcommand{\Lra}{\Longrightarrow}

\begin{document}
\title{Cut-by-curves criterion for the overconvergence of $p$-adic 
differential equations}
\author{Atsushi Shiho
\footnote{
Graduate School of Mathematical Sciences, 
University of Tokyo, 3-8-1 Komaba, Meguro-ku, Tokyo 153-8914, JAPAN. 
E-mail address: shiho@ms.u-tokyo.ac.jp \, 
Mathematics Subject Classification (2000): 12H25.}}
\date{}
\maketitle

\begin{abstract}
In this paper, we prove a `cut-by-curves criterion' for 
the overconvergence of integrable connections on 
certain rigid analytic spaces and certain varieties over $p$-adic fields. 
\end{abstract}

\tableofcontents

\section*{Introduction}

Let $K$ be a complete discrete valuation field of mixed characteristic 
$(0,p)$ with ring of integers $O_K$ and residue field $k$. 
Assume we are given an open immersion $\cX \hra \ol{\cX}$ of $p$-adic 
formal schemes separated and smooth over $\Spf O_K$ such that the 
complement is a relative simple normal crossing divisor and let 
$X \hra \ol{X}$ be the special fiber of $\cX \hra \ol{\cX}$. 
(In this paper, we call such a pair $(\cX,\ol{\cX})$ 
{\it a formal smooth pair} and call the pair $(X,\ol{X})$ 
{\it the special fiber} of $(\cX,\ol{\cX})$.) Then 
we have an admissible open immersion $\cX_K \hra \ol{\cX}_K$ of 
associated rigid spaces over $K$ and an open immersion 
$]\ol{X}-X[_{\ol{\cX}} \hra \ol{\cX}_K$ from the tubular neighborhood of 
$\ol{X}-X$ in $\ol{X}_K$. A strict neighborhood of $\cX_K$ in $\ol{\cX}_K$ 
is an admissible open set $\fU$ of $\ol{\cX}_K$ containing $\cX_K$ 
such that $\{\fU,]\ol{X}-X[_{\ol{\cX}}\}$ forms an admissible covering 
of $\ol{\cX}_K$. Let $\MIC((\cX_K,\ol{\cX}_K)/K)$ be the category of 
pairs $(\fU,(E,\nabla))$ consisting of a strict neighborhood 
$\fU$ of $\cX_K$ in $\ol{\cX}_K$ and a $\nabla$-module
($=$locally free module of finite rank endowed with an 
integrable connection) $(E,\nabla)$ on $\fU$ over $K$, whose set of 
morphisms is defined by $\Hom((\fU,(E,\nabla)),(\fU',(E',\nabla'))) 
:= \varinjlim_{\fU''}\Hom((E,\nabla)|_{\fU''},(E',\nabla')|_{\fU''})$, 
where $\fU''$ runs through strict neighborhoods of $\cX_K$ in 
$\ol{\cX}_K$ contained in $\fU\cap\fU'$. We 
call an object in $\MIC((\cX_K,\ol{\cX}_K)/K)$ a $\nabla$-module on 
a strict neighborhood of $\cX_K$ in $\ol{\cX}_K$ by abuse of 
terminology, and we will often denote it simply by $(E,\nabla)$ in the 
following. 
We say that a $\nabla$-module on 
a strict neighborhood of $\cX_K$ in $\ol{\cX}_K$ is {\it overconvergent}
if it comes from an overconvergent isocrystal on $(X,\ol{X})/K$. 
In this paper, as the first main theorem, 
we prove a `cut-by-curves criterion' for the overconvergence 
of integrable connections on strict neighborhoods of $\cX_K$ 
in $\ol{\cX}_K$. \par 
A morphism $f:(\cY,\ol{\cY}) \lra (\cX,\ol{\cX})$ of formal smooth pairs 
is a morphism $f:\ol{\cY} \lra \ol{\cX}$ satisfying $f(\cY) \subseteq \cX$ 
and it is called strict if $f^{-1}(\cX)=\cY$. $f$ is called a (locally) closed 
immersion if so is the morphism $f:\ol{\cY} \lra \ol{\cX}$. 
If $f:(\cY,\ol{\cY}) \lra (\cX,\ol{\cX})$ is a morphism of formal smooth 
pairs and if $(E,\nabla)$ is a $\nabla$-mobule on a strict neighborhood 
of $\cX_K$ in $\ol{\cX}_K$, we can define in natural way the pull-back 
$f^*(E,\nabla)$ of $(E,\nabla)$ by $f$, which is a $\nabla$-module on 
a strict neighborhood of $\cY_K$ in $\ol{\cY}_K$. Then our first 
main theorem is described as follows: 

\begin{thm}\label{mainthm}
Let $K, k$ be as above and assume that $k$ is uncountable. 
Let $(\cX,\ol{\cX})$ be a formal smooth pair and let $(E,\nabla)$ be 
a $\nabla$-module on a strict neighborhood of $\cX_K$ in $\ol{\cX}_K$. 
Then the following are equivalent$:$ 
\begin{enumerate}
\item 
$(E,\nabla)$ is overconvergent. 
\item 
For any strict locally closed immersion $i:(\cY,\ol{\cY}) \hra 
(\cX,\ol{\cX})$ of formal smooth pairs with $\dim \ol{\cY} = 1$, 
$i^*(E,\nabla)$ is overconvergent. 
\end{enumerate}
\end{thm} 

Note that, in Theorem \ref{mainthm}, the implication 
(1)$\,\Longrightarrow\,$(2) is easy: Indeed, 
if we denote the morphism induced by $i$ on special fibers by 
$i_0:(Y,\ol{Y}) \lra (X,\ol{X})$ and if $(E,\nabla)$ comes from 
an overconvergent isocrystal $\cE$ on $(X,\ol{X})/K$, 
$i^*(E,\nabla)$ comes from the overconvergent isocrystal 
$i_0^*\cE$ on $(Y,\ol{Y})/K$ and hence it is overconvergent. 
So what we should prove is the implication 
(2)$\,\Longrightarrow\,$(1). We prove it by using 
Kedlaya's result (\cite{kedlayamore}) on etale covers of smooth $k$-varieties, 
the notion of intrinsic generic radius of convergence of 
$\nabla$-modules on polyannuli due to Kedlaya-Xiao (\cite{kedlayaxiao}) and 
some techniques developped in \cite{cc}. \par 
The second main theorem is an algebraic variant of the theorem above. 
Assume we are given an open immersion $\X \hra \ol{\X}$ 
of smooth schemes over $\Spec O_K$ such that the 
complement is a relative simple normal crossing divisor.  
(In this paper, we call such a pair $(\X,\ol{\X})$ 
{\it a smooth pair over} $O_K$.) Denote the generic fiber of $\X$ by 
$\X_K$ and let $\MIC((\X_K/K)$ be the category of $\nabla$-module
($=$locally free module of finite rank endowed with an 
integrable connection) $(E,\nabla)$ on $\X_K$ over $K$ in algebraic sense. 
Let $\X^{\an}_K$ be the rigid analytic space associated to the 
$K$-scheme $\X_K$ and let $\MIC((\X^{\an}_K/K)$ 
be the category of $\nabla$-modules 
on $\X^{\an}_K$ over $K$ in analytic sense. Also, let 
$\wh{\X}, \wh{\ol{\X}}$ be the $p$-adic completion of $\X, \ol{\X}$, 
respectively. Then $\X^{\an}_K \cap \wh{\ol{\X}}_K$ 
is a strict neighborhhod of $\wh{\X}_K$ in 
$\wh{\ol{\X}}_K$. Hence we have the functors 
$$ 
\begin{aligned}
& \MIC(\X_K/K) \lra  \MIC(\X^{\an}_K/K) \lra 
\MIC((\wh{\X}_K,\wh{\ol{\X}}_K)/K); \\ 
& \hspace{0.7cm} (E,\nabla) \hspace{0.5cm}\longmapsto\hspace{0.3cm}  
(E^{\an},\nabla^{\an}) \hspace{0.3cm}\longmapsto
(\X^{\an}_K \cap \wh{\ol{\X}}_K,(E^{\an},\nabla^{\an})), 
\end{aligned} 
$$ 
where the first one is the analytification. We call an object $(E,\nabla)$ in 
$\MIC(\X_K/K)$ {\it overconvergent} if the associated object 
$(\X^{\an}_K \cap \wh{\ol{\X}}_K, (E^{\an},\nabla^{\an}))$ in 
$\MIC((\wh{\X}_K,\wh{\ol{\X}}_K)/K)$ is overconvergent. \par 
We can define the notion of a (strict) morphism 
$f:(\Y,\ol{\Y}) \lra (\X,\ol{\X})$ of smooth pairs over $O_K$ 
in natural way and it is called a (locally) closed 
immersion if so is the morphism $f:\ol{\Y} \lra \ol{\X}$. 
As in the case of formal smooth pairs, we can define the pull-back of a 
$\nabla$-module on $\X_K$ by $f$, which is a $\nabla$-module on $\Y_K$. 
Then our second main theorem is described as follows: 

\begin{thm}\label{mainthm2}
Let $K, k$ be as above and assume that $k$ is uncountable. 
Let $(\X,\ol{\X})$ be a formal smooth pair such that 
$\ol{\X}$ is projective over $O_K$ and let $(E,\nabla)$ be 
a $\nabla$-module on $\X_K$. 
Then the following are equivalent$:$ 
\begin{enumerate}
\item 
$(E,\nabla)$ is overconvergent. 
\item 
For any strict locally closed immersion $i:(\Y,\ol{\Y}) \hra 
(\X,\ol{\X})$ of smooth pairs over $O_K$ with $\dim (\ol{\Y}/O_K) = 1$, 
$i^*(E,\nabla)$ is overconvergent. 
\end{enumerate}
\end{thm} 

Since the implication (1)$\,\Longrightarrow\,$(2) is obvious as in the case 
of Theorem \ref{mainthm}, it suffices to prove the implication 
(2)$\,\Longrightarrow\,$(1). We prove it by reducing to a refined version of 
Theorem \ref{mainthm} (see Remark \ref{refined}): To do this, 
we prove a partial generalization of 
Kedlaya's result (\cite[Theorem 2]{kedlayamore}) 
on etale covers of smooth $k$-varieties to the case of smooth formal schemes 
and smooth schemes over $O_K$. \par 
The author is partly supported by Grant-in-Aid for Young Scientists (B), 
the Ministry of Education, Culture, Sports, Science and Technology of 
Japan and JSPS Core-to-Core program 18005 whose representative is 
Makoto Matsumoto. \par 

\section*{Convention}
Throughout this paper, $K$ is 
a complete discrete valuation field of mixed characteristic 
$(0,p)$ with ring of integers $O_K$ and residue field $k$. 
Let $|\cdot|: K \lra \R_{\geq 0}$ be a fixed valuation of $K$ and 
let $\Gamma^*$ be $\sqrt{|K^{\times}|} \cup \{0\}$. 
For a $p$-adic formal scheme $\cX$ topologically of finite type over $O_K$, 
we denote the associated rigid space by $\cX_K$. 
A $k$-variety means a reduced separated scheme of finite type 
over $k$. A closed point in a smooth $k$-variety $X$ is called a separable 
closed point if its residue field is separable over $k$. 
For a $p$-adic smooth formal scheme $\cX$ over $\Spf O_K$ 
(resp. a smooth scheme $\X$ over $\Spec O_K$) with 
special fiber $X$ and a separable closed point $x$ of $X$, 
a lift of $x$ in $\cX$ (resp. $\X$) is a closed sub $p$-adic formal scheme 
$\ti{x} \hra \cX$ which is etale over $\Spf O_K$ 
(resp. a closed subscheme 
$\ti{x} \hra \X$ which is etale over $\Spec O_K$) with special fiber $x$. 
(Note that there always exists a lift of $x$.) \par 
We use freely the notion concerning overconvergent 
isocrystals. For detail, see \cite{berthelotrig} and 
\cite[\S 2]{kedlayaI}. See also 
Propositions \ref{octaylor}, \ref{cov} in the text. 

\section{Proof of the first main theorem}

In this section, we give a proof of the first main theorem (Theorem 
\ref{mainthm}). 
First, let us recall the basic definition concerning overconvergence and 
recall a concrete description of overconergence for $\nabla$-modules, 
which is due to Berthelot (\cite[2.2.13]{berthelotrig}, 
\cite[2.5.6--8]{kedlayaI}). \par 
For a formal smooth pair $(\cX,\ol{\cX})$ with special fiber $(X,\ol{X})$, 
let us denote the category of overconvergent isocrystals on $(X,\ol{X})/K$ 
by $I^{\d}((X,\ol{X})/K)$. Then we have a fully-faithful functor 
$$ \Phi_{(\cX,\ol{\cX})}: 
I^{\d}((X,\ol{X})/K) \lra \MIC((\cX_K,\ol{\cX}_K)/K) $$ 
which is functorial with respect to $(\cX,\ol{\cX})$. As is written in 
the introduction, we say that an object $(E,\nabla)$ in 
$\MIC((\cX_K,\ol{\cX}_K)/K)$ is overconvergent if it is in the essential 
image of $\Phi_{(\cX,\ol{\cX})}$. \par 
In this paper, we need a 
concrete description of overconergence for the objects 
in $\MIC((\cX_K,\ol{\cX}_K)/K)$, using the Taylor series. So we will 
recall it. Let $(\cX,\ol{\cX})$ be a formal smooth pair and 
assume, for the moment, the following condition $(*)$: \\
\quad \\
 $(*)$: \,\,\,\, 
$\ol{\cX}$ is affine, $\ol{\cX}-\cX$ is a union of smooth divisors 
$\cD_i \,(1 \leq i \leq n)$ which are defined as the zero locus 
of some $t_i \in \Gamma(\ol{\cX},\cO_{\cX})\,(1 \leq i \leq m)$, 
and $\Omega^1_{\ol{\cX}_K}$ is freely generated by $dt_i \,(1 \leq i \leq 
n+m)$ for $t_1,...,t_n$ as before and some $t_{n+1}, ..., t_{n+m} \in 
\Gamma(\ol{\cX},\cO_{\cX})$. \\
\quad \\
With the assumption of $(*)$, we denote the induced differential 
operator $\pa/\pa t_i$ (which operates on $\nabla$-modules on strict 
neighborhoods of $\cX_K$ in $\ol{\cX}_K$) 
simply by $\pa_i$. Also, for $\lam \in (0,1] \cap \Gamma^*$, we put 
$$ 
\fU_{\lam} := \{x \in \ol{\cX}_K\,\vert\, |t_i(x)| \geq \lam \,\, 
\text{for all $1 \leq i \leq n$}\}. $$
Then we can give a 
concrete description of overconergence in terms of Taylor series in 
the following way: 

\begin{prop}[{\cite[2.2.13]{berthelotrig}, 
\cite[2.5.6--8]{kedlayaI}}]\label{octaylor}
Let $(\cX,\ol{\cX})$ be as above and let $(E,\nabla)$ be a 
$\nabla$-module on a strict neighborhood $\fV$ of $\cX_K$ in $\ol{\cX}_K$. 
Fix a set of generators $(e_{\alpha})_{\alpha}$ of $\Gamma(\fV,E)$. 
Then $(E,\nabla)$ is overconvergent if and only if the following 
condition is satisfied$:$ For each $\eta \in (0,1)\cap\Gamma^*$, 
there exists $\lam_0 \in (0,1)\cap\Gamma^*$ such that for any 
$\lam \in [\lam_0,1)\cap\Gamma^*$ and any $\alpha$, we have 
$\fU_{\lam} \subseteq \fV$ and the multi-sequence 
$$\left\{ \left\| \dfrac{1}{i_1!\cdots i_{n+m}!}\pa^{i_1}_1\cdots
\pa^{i_{n+m}}_{n+m} 
(e_{\alpha}) \right\| \eta^{i_1+\cdots +i_{n+m}} \right\}_{i_1,...,i_{n+m}}$$ 
tends to zero as $i_1,...,i_{n+m} \to\infty$, where $\| \cdot \|$ denotes 
any $p$-adic Banach norm on $\Gamma(\fU_{\lam},E)$ induced by the 
affinoid norm on $\Gamma(\fU_{\lam},\cO)$. 
\end{prop} 

\begin{rem} 
In \cite[2.5.6]{kedlayaI}, the set 
$\fU'_{\lam} := 
\{x \in \ol{\cX}_K\,\vert\, |\prod_{i=1}^nt_i(x)| \geq \lam\}$ 
is used instead of $\fU_{\lam}$. However, this does not cause any problem 
because we have the inclusions $\fU'_{\lam} \subseteq \fU_{\lam} 
\subseteq \fU'_{\lam^n}$. 
\end{rem}

Next let us consider the case where the given formal smooth pair 
$(\cX,\ol{\cX})$ does not necessarily satisfy the condition $(*)$. 
Noting the fact that the functor $\Phi_{(\cX,\ol{\cX})}$ is fully-faithful 
and functorial and noting the descent property of 
the categories $I^{\d}((X,\ol{X})/K)$, $\MIC((\cX_K,\ol{\cX}_K)/K)$  
with respect to Zariski 
coverings of $\ol{\cX}$, we see easily the following (we omit the proof): 

\begin{prop}\label{cov}
Let $(\cX,\ol{\cX})$ be a formal smooth pair and let $(E,\nabla)$ be a 
$\nabla$-module on a strict neighborhood of $\cX_K$ in $\ol{\cX}_K$. 
Then $(E,\nabla)$ is overconvergent if and only if there exists 
an open covering $\ol{\cX} = \bigcup_{\alpha} \ol{\cX}_{\alpha}$ such 
that $(\cX_{\alpha},\ol{\cX}_{\alpha})$ $($where $\cX_{\alpha} := 
\ol{\cX}_{\alpha} \cap \cX$ $)$ satisfies the condition $(*)$ and 
that the restrction of $(E,\nabla)$ to a strict neighborhood of 
$\cX_{\alpha,K}$ in $\ol{\cX}_{\alpha,K}$ is overconvergent. 
\end{prop}

By Proposition \ref{cov}, we see that, to prove the overconvergence of 
$(E,\nabla)$, it suffices to check the property described in Proposition 
\ref{octaylor} locally on $\ol{\cX}$. \par 
Now we give a first step of the proof of Theorem \ref{mainthm}. 

\begin{prop}\label{firstred}
Theorem \ref{mainthm} is true if it is true for formal smooth pairs 
of the form 
$(\fG_{m,O_K}^n \times \fAf_{O_K}^m, 
\fAf_{O_K}^{n+m})\,(n,m\in\N)$. 
\end{prop}

\begin{proof}
Let $(\cX,\ol{\cX})$ be a formal smooth pair and let 
$(E,\nabla)$ be a $\nabla$-module on a strict neighborhood of 
$\cX_K$ in $\ol{\cX}_K$ satisfying the condition (2) in Theorem 
\ref{mainthm}. Let $(X,\ol{X})$ be the special fiber of $(\cX,\ol{\cX})$. 
First we prove the following claim: \\
\quad \\
{\bf claim.} \,\,\, 
For any closed point $x$ of $\ol{X}$, there exist 
open immersions $\ol{\cU}_x \hra \ol{\cX}_x \hra \ol{\cX}$ containing $x$, 
an open sub formal scheme $\cX_x$ of $\ol{X}_x$ and a diagram 
of formal smooth pairs 
$$ (\cX,\ol{\cX}) \os{j}{\lla} (\cX_x, \ol{\cX}_x) \os{f}{\lra} 
(\fG_{m,k}^n \times \Af_{k}^m, \fAf_{k}^{n+m}) $$ 
(where $j$ is induced by the open immersion $\ol{\cX}_x \hra 
\ol{\cX}$) for some $n,m$ 
such that $f$ is a strict finite etale morphism and that 
the morphism $(\ol{\cU}_x \cap \cX_x, \ol{\cU}_x) \lra (\ol{\cU}_x \cap \cX, 
\ol{\cU}_x)$ induced by $j$ is an isomorphism. \\

\begin{proof}[Proof of claim] 
Let us put $\cD := \ol{\cX} - \cX$ and let $D$ be the special fiber 
of $\cD$, which is a simple normal crossing divisor in $\ol{X}$. 
Let $D_{\ni x} = \bigcup_{i=1}^n D_i$ be the union of irreducible 
components of $D$ containing $x$ and let $D_{\not\ni x}$ be the 
union of irreducible 
components of $D$ not containing $x$. By applying 
\cite[Theorem 2]{kedlayamore} to $\ol{X} - D_{\not\ni x}$ and the simple 
normal crossing divisor $D - D_{\not\ni x} = D_{\ni x} - D_{\not\ni x}$ 
on it, we see that there exists an open subscheme $\ol{X}_x$ in 
$\ol{X} - D_{\not\ni x}$ containing $x$ and a finite etale morphism 
$f_0: \ol{X}_x \lra \Af^{n+m}_k$ for some $m$ such that, for 
$1 \leq i \leq n$, $f_0(D_i \cap \ol{X}_x)$ is contained in 
the $i$-th coordinate hyperplane $H_i$ of $\Af^{n+m}_k$. Then 
$D_i \cap \ol{X}_x \subseteq f_0^{-1}(H_i)$ 
is an open and closed immersion, and so 
$D \cap \ol{X}_x = D_{\ni x} \cap \ol{X}_x$ is a sub simple normal 
crossing divisor of $\bigcup_{i=1}^n f_0^{-1}(H_i)$ such that the 
complementary divisor does not contain $x$. \par 
Let $f': \ol{\cX}'_x \lra 
\fAf^{n+m}_{O_K}$ be a finite etale morphism of formal schemes lifting 
$f_0$ and let $\cH_i \subseteq \fAf^{n+m}_{O_K}\,(1 \leq i \leq n)$ be 
the $i$-th coordinate hyperplane. Then 
$\bigcup_{i=1}^n {f'}^{-1}(\cH_i)$ is a relative simple normal crossing 
divisor in $\ol{\cX}'_x$ lifting $\bigcup_{i=1}^n f_0^{-1}(H_i)$ and so 
there exists a sub relative simple normal crossing divisor of 
 $\bigcup_{i=1}^n f_0^{-1}(H_i)$ lifting $D \cap \ol{X}_x$, which we 
denotes by $\cD'_x$. \par 
On the other hand, let $\ol{\cX}_x$ be the open sub formal scheme of 
$\ol{\cX}$ lifting $\ol{X}_x$ and put $\cD_x := \cD \cap \ol{\cX}_x$. 
Then both $(\ol{\cX}'_x, \cD'_x)$, $(\ol{\cX}_x,\cD_x)$ lifts 
$(\ol{X}_x, D \cap \ol{X}_x)$. Considering them as fine log formal 
schemes log smooth over $\Spf O_K$, we see by \cite{kato} that they are 
isomorphic. In particular, we have the isomorphism 
$$ \iota: (\cX \cap \ol{\cX}_x, \ol{\cX}_x) = 
(\ol{\cX}_x - \cD_x, \ol{\cX}_x) \os{=}{\lra} 
(\ol{\cX}'_x - \cD'_x, \ol{\cX}'_x) $$
and $\cD_x$ is a sub relative simple normal crossing divisor of 
$\bigcup_{i=1}^n \iota^{-1}{f'}^{-1}(\cH_i)$. Let us denote the 
complementary divisor by $\cC$. \par 
Now let us put 
$\cX_x := \ol{\cX}_x - \bigcup_{i=1}^n \iota^{-1}{f'}^{-1}(\cH_i)$, 
let $f$ be the composite 
\begin{align*}
(\cX_x,\ol{\cX}_x) & \os{\cong}{\lra} (\ol{\cX}'_x - 
\bigcup_{i=1}^n {f'}^{-1}(\cH_i), \ol{\cX}'_x) \\ 
& \os{f'}{\lra} 
(\fAf^{n+m}_{O_K} - \bigcup_{i=1}^n \cH_i, \fAf^{n+m}_{O_K}) = 
(\fG_{m,O_K}^n \times \fAf^m_{O_K}, \fAf^{n+m}_{O_K}) 
\end{align*}
(where the first morphism is induced by $\iota$ and the second one is 
induced by $f'$) and let $j$ be the composite 
$$ (\cX_x,\ol{\cX}_x) \lra (\ol{\cX}_x - \cD_x, \ol{\cX}_x) 
= (\cX \cap \ol{\cX}_x, \ol{\cX}_x) \os{\subset}{\lra} (\cX,\ol{\cX}) $$
induced by the canonical inclusions. Also, let us put 
$\ol{\cU}_x := \ol{\cX}_x - \cC$. Then we see easily that they
satisfy the required properties. So the proof of the claim is finished. 
\end{proof}

Let us return to the proof of the proposition. For any closed point 
$x$ in $\ol{X}$, let us choose 
open immersions $\ol{\cU}_x \hra \ol{\cX}_x \hra \ol{\cX}$ containing $x$, 
an open sub formal scheme $\cX_x$ of $\ol{X}_x$ and morphisms $j,f$ as in 
the above claim. Let us note that the following are equivalent: 
\begin{enumerate}
\item[(a)] $(E,\nabla)$ is overconvergent. 
\item[(b)] For any closed point $x$ in $\ol{X}$, the restriction of 
$(E,\nabla)$ to $\MIC((\cX_{x,K},\ol{\cX}_{x,K})/K)$ is overconvergent. 
\item[(c)] For any closed point $x$ in $\ol{X}$, the restriction of 
$(E,\nabla)$ to $\MIC((\cU_{x,K} \cap \cX_K ,\ol{\cU}_{x,K})\allowbreak 
/K) = 
\MIC((\cU_{x,K} \cap \cX_{x,K} ,\ol{\cU}_{x,K})/K)$ is overconvergent. 
\end{enumerate}
Indeed, (a)$\,\Lra\,$(b)$\,\Lra\,$(c) is obvious, and since each 
$(\cU_{x} \cap \cX_{x} ,\ol{\cU}_{x})$ satisfies the condition $(*)$, 
we have (c)$\,\Lra\,$(a) by Proposition \ref{cov}. By this equivalence, 
we can replace $(\cX, \ol{\cX})$ by $(\cX_x, \ol{\cX}_x)$ to prove the 
overconvergence of $(E,\nabla)$. So we may assume that there exists a 
strict finite etale morphism $f: (\cX, \ol{\cX}) \lra 
(\fG_{m,k}^n \times \fAf_{k}^m, \fAf_{k}^{n+m})$ to prove 
the proposition. \par 
In the following, we put 
$(\cX_0,\ol{\cX}_0) := (\fG_{m,O_K}^n \times \fAf_{O_K}^m, \fAf_{O_K}^{n+m}), 
(X_0,\ol{X}_0) := (\G_{m,k}^n \times \Af_{k}^m, \Af_{k}^{n+m})$ and 
let $f_K: \ol{\cX}_K \lra 
\ol{\cX}_{0,K}$ be the morphism of rigid spaces associated to $f$.  
Let $t_1,...,t_n$ be the coordinate of $\fG_{m,O_K}^n$ and let us define 
$\fU_{0,\lam} \subseteq \ol{\cX}_{0,K}, \fU_{\lam} \subseteq \ol{\cX}_K$ 
by 
$$ 
\fU_{0,\lam} := \{x \in \ol{\cX}_{0,K}\,\vert\, |t_i(x)| \geq \lam \,\, 
\text{for all $1 \leq i \leq n$}\}, $$
$$ 
\fU_{\lam} := \{x \in \ol{\cX}_K\,\vert\, |t_i(x)| \geq \lam \,\, 
\text{for all $1 \leq i \leq n$}\}. $$
Then $f_K$ induces a finite etale morphism 
$f_K: \fU_{\lam} \lra \fU_{0,\lam}$ between affinoid rigid spaces, 
that is, $\Gamma(\fU_{\lam},\cO)$ is finite etale algebra 
over $\Gamma(\fU_{0,\lam},\cO)$. \par 
For some $\lam$, $(E,\nabla)$ is defined on $\fU_{\lam}$. 
Then we can define the push-forward $f_*(E,\nabla)$ of $(E,\nabla)$ by 
$f$, which is a $\nabla$-module on $\fU_{0,\lam}$ by 
$$ f_*(E,\nabla) := (f_*E, f_*E \os{f_*\nabla}{\lra} 
f_*(E \otimes \Omega^1_{\fU_{\lam}}) = 
f_*E \otimes \Omega^1_{\fU_{0,\lam}}). $$
It is easy to see the following properties: 
\begin{enumerate}
\item 
For $\nabla$-modules $(F,\nabla_F)$, $(F',\nabla_{F'})$ on $\fU_{\lam}$,  
there exists a canonical isomorphism 
$$ f_*: \Hom_{\fU_{\lam}}((F,\nabla_F),(F',\nabla_{F'})) \os{=}{\lra} 
\Hom_{\fU_{0,\lam}}(f_*(F,\nabla_F),f_*(F',\nabla_{F'})). $$
\item 
For a $\nabla$-module $(F,\nabla_F)$ on $\fU_{0,\lam}$, 
there exists functorially 
the adjunction map $\ad: (F,\nabla_F) \lra f_*f^*(F,\nabla_F)$ 
and the trace map $\tr: f_*f^*(F,\nabla_F) \lra (F,\nabla_F)$ such that 
the composition $\tr \circ \ad$ is equal to the multiplication by 
the 
degree $d$ of $\ol{\cX}$ over $\ol{\cX}_0$.) 
\end{enumerate}
By (2) above, we have the morphisms 
$\ad: f_*(E,\nabla) \lra f_*f^*f_*(E,\nabla)$, 
$\tr: f_*f^*f_*(E, \allowbreak 
\nabla) \lra f_*(E, \nabla)$ with $\tr\circ\ad = d$ 
and by (1), 
we have $\alpha: (E,\nabla) \lra f^*f_*(E,\nabla)$, 
$\beta: f^*f_*(E,\nabla) \lra (E,\nabla)$ with $f_*\alpha=\ad, 
f_*\beta=\tr$. Then we have $\beta \circ \alpha = d$ and so 
$(E,\nabla)$ is a direct summand of $f^*f_*(E,\nabla)$. \par 
Now we prove that $(E,\nabla)$ is overconvergent if and only if 
$f_*(E,\nabla)$ is overconvergent. Indeed, if $f_*(E,\nabla)$ is 
overconvergent, then so is $f^*f_*(E,\nabla)$ and since 
$(E,\nabla)$ is a direct summand of $f^*f_*(E,\nabla)$, 
$(E,\nabla)$ is also overconvergent (that is, it satisfies the 
condition given in Proposition \ref{octaylor}.) On the other hand, 
by \cite[5.1]{tsuzuki}, 
the push-forward functor $f_*: I^{\d}((X,\ol{X})/K) \lra 
I^{\d}((X_0,\ol{X}_0)/K)$ 
for overconvergent isocrystals such that the diagram 
\begin{equation*}
\begin{CD}
I^{\d}((X,\ol{X})/K) @>{\Phi_{(\cX,\ol{\cX})}}>> 
\MIC((\cX_K,\ol{\cX}_K)/K) \\ 
@V{f_*}VV @V{f_*}VV \\ 
I^{\d}((X_0,\ol{X}_0)/K) 
@>{\Phi_{(\cX_0,\ol{\cX}_0)}}>> 
\MIC(((\cX_{0,K},\ol{\cX}_{0,K})/K) \\ 
\end{CD}
\end{equation*}
is commutative, where $f_*$ in the right vertical arrow is 
the functor $(F,\nabla_F) \mapsto f_*(F,\nabla_F)$ defined above. 
This diagram implies that $f_*(E,\nabla)$ is overconvergent 
if so is $(E,\nabla)$. \par 
Now let us prove the proposition. Let $(E,\nabla)$ be as in the 
beginning of the proof. Let 
$i_0: (\cY_0,\ol{\cY}_0) \hra (\cX_0,\ol{\cX}_0)$ 
be any locally closed immersion of formal smooth pairs 
with $\dim \ol{\cY}_0 =1$ and let 
$i:(\cY,\ol{\cY}) \hra (\cX,\ol{\cX})$ be the pull-back of $i_0$ by $f$. 
Then $i^*(E,\nabla)$ is overconvergent by assumption, and by the argument 
similar to the previous paragraph ($f$ replaced by $f|_{\ol{\cY}}$), 
$(f|_{\ol{\cY}})_*i^*(E,\nabla) = i_0^*f_*(E,\nabla)$ is also 
overconvergent. Since $i_0$ was arbitrary, Theorem \ref{mainthm} for 
$(\cX_0,\ol{\cX}_0)$ (which we assumed) 
implies the overconvergence of $f_*(E,\nabla)$ and hence 
$(E,\nabla)$ is also overconvergent. So we have proved the proposition. 
\end{proof}

Next we recall the definition of intrinsic generic radius of 
convergence of $\nabla$-modules on polyannuli, which is due to 
Kedlaya-Xiao \cite{kedlayaxiao}. 
Let $L$ be a field containing $K$ 
complete with respect to a norm 
(denoted also by $|\cdot|$) which extends the given absolute 
value of $K$. A subinterval $I \subseteq [0,\infty)$ is called 
aligned if any endpoint of $I$ at which it is closed is contained in 
$\Gamma^*$ and for an aligned interval $I$, we 
define the rigid space $A^n_L(I)$ by  
$A^n_L(I) := \{(t_1, ..., t_n) \in {\Bbb A}^{n,\an}_L 
\,\vert\, \forall i, |t_i| \in I \}.$ \par 
For a formal smooth pair $(\cX_0, \ol{\cX}_0) := 
(\fG^n_{m,O_K} \times \fAf^m_{O_K}, \fAf^m_{O_K})$ and 
the strict neighborhood 
$\fU_{0,\lam}$ of $\cX_{0,K}$ $\ol{\cX}_{0,K}$ defined in the proof of 
Proposition \ref{firstred}, we have 
$\cX_{0,K} = A^n_K[1,1] \times A^m_K[0,1], 
\ol{\cX}_{0,K} = A^{n+m}[0,1]$ and 
$\fU_{0,\lam} = A^n[\lam,1] \times A^m[0,1]$. So it is important 
to study $\nabla$-modules on these polyannuli. \par 
So let $L$ be as above, let $n,m \in \N$, let $\lam \in [0,1]\cap\Gamma^*$ 
and let $(E,\nabla)$ be a $\nabla$-module on $A_L^n[\lam,1] \times 
A^m_L[0,1]$. Let us denote the coordinate of $A_L^n[\lam,1] \times 
A^m_L[0,1]$ by $t_1,...,t_{n+m}$ and put $\pa_i := 
\pa/\pa t_i \,(1 \leq i \leq n+m)$. For 
$\rrho := (\rho_i) \in [\lam,1]^n \times [0,1]^m$, let $L(t)_{\rrho}$ 
be the completion of $L(t) := L(t_1,...,t_{n+m})$ by 
$\rrho$-Gauss norm. Then $L(t)_{\rrho}$ is endowed with derivations
 $\pa_i \,(1 \leq i \leq n+m)$ and so we can define the 
spectral norm of $\pa_i$ on $L(t)_{\rrho}$, which we denote by 
$|\pa_i|_{\rrho,L,\sp}$. It is easy to see that 
$|\pa_i|_{\rrho,L,\sp} = p^{-1/(p-1)}\rho_i^{-1}$ in our case. 
On the other hand, $(E,\nabla)$ induces 
a differential module $E_{\rrho}$ on $L(t)_{\rrho}$ with respect to 
$\pa_i\,(1 \leq i \leq n+m)$ and so we can define the spectral norm 
of $\pa_i$ on $E_{\rrho}$, which we denote by $|\pa_i|_{\rrho,E,\sp}$. 
Then we define the intrinsic generic radius of convergence 
$IR(E,\rrho)$ of $E$ with radius $\rrho$ by 
$$ IR(E,\rrho) 
= \min_{i} \{|\pa_i|_{\rrho,L,\sp}/|\pa_i|_{\rrho,E,\sp}\} 
= \min_{i} \{p^{-1/(p-1)}\rho_i^{-1}|\pa_i|^{-1}_{\rrho,E,\sp}\}. $$
Since it is known (\cite[6.2.4]{kedlayabook}) that we always have 
$p^{-1/(p-1)}\rho_i^{-1}|\pa_i|^{-1}_{\rrho,E,\sp} \leq 1$, 
$IR(E,\rrho) \leq 1$. If $\e := (\e_1,...,\e_{\mu})$ is a 
basis of $E_{\rrho}$ and $G_{i,n}$ is the matrix expression 
of the operator $\pa_i^n$ on $E_{\rrho}$ with respect to the basis $\e$, 
we have the equality 
$$ p^{-1/(p-1)}\rho_i^{-1}|\pa_i|^{-1}_{\rrho,E,\sp} 
= \min \{1, p^{-1/(p-1)}\rho_i^{-1}
\ul{\lim}_{n\to\infty}|G_{i,n}|^{-1/n}_{\rrho} \}, $$
where $|\cdot|_{\rrho}$ denotes the maximum of the $\rrho$-Gauss norms of 
the entries. (See \cite[6.2.5]{kedlayabook}.) \par 
The following is one of the main results of \cite{kedlayaxiao}, which we 
use later. 

\begin{prop}[Kedlaya-Xiao]\label{kedlayaxiaoprop}
The function $\rrho \mapsto IR(E,\rrho)$ is continuous.
\end{prop}

\begin{cor}\label{kedlayaxiaocor}
Fix $1 \leq i \leq n+m$ and 
for $\rho \in [0,1]$, let us denote $(1,...,1,\os{i}{\check{\rho}},
1,...,1)$ by $\rrho$. Then the function 
$\rho \mapsto p^{-1/(p-1)}\rho^{-1}|\pa_i|^{-1}_{\rrho,E,\sp}$ is 
continuous. 
\end{cor}

\begin{proof}
Let $\ti{L}$ be the completion of $L(t_1,...,\check{t}_i,...,t_{n+m})$ 
with respect to $(1,...,1)$-Gauss norm. Then $(E,\nabla)$ induces 
the $\nabla$-module $(\ti{E},\ti{\nabla})$ on $A^1_{\ti{L}}[\lam,1]$ 
(when $1 \leq i \leq n$) or on $A^1_{\ti{L}}[0,1]$ 
(when $n+1 \leq i \leq n+m$), and we have the equality 
$|\pa_i|_{\rrho,E,\sp} = |\pa_i|_{\rho,\ti{E},\sp}$. Hence we have 
$p^{-1/(p-1)}\rho^{-1}|\pa_i|^{-1}_{\rrho,E,\sp} = R(\ti{E},\rho)$ and hence 
it is continuous by Proposition \ref{kedlayaxiaoprop}. 
\end{proof}

Let $(\cX_0, \ol{\cX}_0)$ be $(
(\fG^n_{m,O_K} \times \fAf^m_{O_K}, \fAf^m_{O_K})$ and define 
the strict neighborhhood $\fU_{0,\lam}$ of $\cX_{0,K}$ in 
$\ol{\cX}_{0,K}$ as before. For $\nabla$-modules on 
$\fU_{0,\lam,K} = A^n_K[\lam,1] \times A^m_K[0,1]$, 
the overconvergence is described in terms of intricsic generic radius 
of convergence, as follows: 

\begin{prop}\label{oc-ir}
Let $(\cX_0,\ol{\cX}_0)$, $\fU_{0,\lam}$ be as above and let $(E,\nabla)$ 
be a $\nabla$-module on $\fU_{0,K}$. Then the following are equivalent$:$ 
\begin{enumerate}
\item 
$(E,\nabla)$ is overconvergent. 
\item 
$IR(E,\1)=1$, where $\1 := (1,...,1)$. 
\end{enumerate}
\end{prop}

\begin{proof}
In the proof, we will use multi-index notation as follows: For 
$j:=(j_1,...,j_{n+m}) \in \N^{n+m}$, we put 
$|j| := j_1 + \cdots + j_{n+m}$, 
$j! :=j_1!\cdots j_{n+m}!$ and 
$\pa^j := \pa_1^{j_1} \cdots \pa_{n+m}^{j_{n+m}}$. 
Fix a set of generators $(e_{\alpha})_{\alpha}$ in $\Gamma(\fU_{0,\lam},E)$. 
Then, by Proposition \ref{octaylor}, $(E,\nabla)$ is overconvergent 
if and only if the following condition is satisfied: 
For each $\eta \in (0,1)\cap\Gamma^*$, 
there exists $\rho_0 \in (\eta,1)\cap\Gamma^*$ such that for any 
$\rho \in [\rho_0,1]\cap\Gamma^*$, we have 
$$ \forall \alpha, \,\, 
\left\| \dfrac{1}{j!}\pa^j(e_{\alpha}) \right\| \eta^{|j|} 
 \to 0 \,\,\,\, (j \in \N^{n+m}, |j| \to \infty),$$ 
where $\| \cdot \|$ denotes 
any $p$-adic Banach norm on $\Gamma(\fU_{0,\rho},E)$ induced by the 
affinoid norm on $\Gamma(\fU_{0,\rho},\cO)$. \par 
First, let us note that the supremum norm on 
$\Gamma(\fU_{0,\rho},\cO)$ gives the same topology as the affinoid norm 
on it by \cite[6.2.4 Theorem 1]{bgr}. Hence we can replace the conclusion 
of the above condition by  
$$\forall \alpha, \,\, 
\left| \dfrac{1}{j!}\pa^j(e_{\alpha}) \right| 
\eta^{|j|} \to 0 \,\,\,\, (j \in \N^{n+m}, |j| \to \infty),$$ 
where $| \cdot |$ 
denotes any $p$-adic Banach norm on $\Gamma(\fU_{0,\rho},E)$ induced by the 
supremum norm on $\Gamma(\fU_{0,\rho},\cO)$. \par 
Next let us note that, if $\rho > \eta^{1/2}$, we have 
\begin{align}
\left| \dfrac{1}{j!}\pa^j 
(\sum_{\alpha} f_{\alpha}e_{\alpha}) \right| \eta^{|j|} 
\leq &
\max_{\alpha,0 \leq j' \leq j} 
\left(
\left| \dfrac{1}{(j-j')!}\pa^{j-j'}
(f_{\alpha}) \right| \cdot 
\left| \dfrac{1}{j'!}\pa^{j'} 
(e_{\alpha}) \right| \eta^{|j|} \right) \label{ineq} \\ 
\leq &
\max_{\alpha,0 \leq j' \leq j} 
\left(|f_{\alpha}| 
\left| \dfrac{1}{j'!}\pa^{j'} 
(e_{\alpha}) \right| \eta^{(|j|+|j'|)/2} 
\right) \nonumber 
\end{align}
for any $f_{\alpha}$'s in $\Gamma(\fU_{0,\rho},\cO)$. Using this, we 
see that we can replace the conclusion 
of the above condition by 
\begin{equation}\label{3}
\left| \dfrac{1}{j!}\pa^j \right| 
\eta^{|j|} \to 0 \,\,\,\, (j \in \N^{n+m}, |j| \to \infty),
\end{equation}
where $| \cdot |$ denotes the operator norm. \par 
Next, note that \eqref{3} implies that 
\begin{equation}\label{4}
\forall i, \,\,\, \left| \dfrac{1}{j!}\pa_i^j \right| 
\eta^{j} \to 0 \,\,\,\, (j \in \N, |j| \to \infty).
\end{equation}
On the other hand, if we assume \eqref{4}, we have 
$\left| \dfrac{1}{j!}\pa_i^j \right| \eta^{j} \leq C$ for some 
constant $C$ independent of $i,j$. Then, 
for any $j =(j_1,...,j_{n+m}) \in \N^{n+m}$, we have an index $i$ with 
$j_i \geq |j|/(n+m)$ and the inequality 
$$ 
\left| \dfrac{1}{j!}\pa^j \right| \eta^{|j|} \leq 
\prod_{l=1}^{n+m} 
\left| \dfrac{1}{j_l!}\pa_l^{j_l} \right| \eta^{j_l} \leq 
C^{n+m-1}
\left| \dfrac{1}{j_i!}\pa_l^{j_i} \right| \eta^{j_i}.
$$ 
This inequality shows that \eqref{4} implies \eqref{3}. Hence 
they are equivalent and so we can replace the conclusion 
of the overconvergence condition by \eqref{4}. \par 
Then, by using the inequality \eqref{ineq} again, we see that we can replace
 the conclusion of the overconvergence condition by 
\begin{equation}
\forall \alpha, \forall i, \,\,\, \left| \dfrac{1}{j!}\pa_i^j(e_{\alpha}) 
\right| 
\eta^{j} \to 0 \,\,\,\, (j \in \N, |j| \to \infty).
\end{equation}
Next, by \cite[3.1.8]{kedlayaI}, we have 
$|x| = \max_{\rrho \in \{\rho,1\}^n \times \{1\}^m} |x|_{\rrho}$ for 
any $x \in \Gamma(\fU_{0,\rho},\cO)$, where $|\cdot|_{\rrho}$ denotes 
the $\rrho$-Gauss norm. Using this, we see that we can replace 
 the conclusion of the overconvergence condition by 
\begin{equation}
\forall \alpha, \forall i, 
\forall \rrho \in \{\rho,1\}^n \times \{1\}^m, 
\,\,\, \left| \dfrac{1}{j!}\pa_i^j(e_{\alpha}) 
\right|_{\rrho,E} 
\eta^{j} \to 0 \,\,\,\, (j \in \N, |j| \to \infty) 
\end{equation}
(where $|\cdot|_{\rrho,E}$ denotes a norm on $E_{\rrho}$ induced by 
the $\rrho$-Gauss norm), and by using the analogue of the 
inequality \eqref{ineq}, we see that it is equivalent to the condition 
\begin{equation}
\forall i, 
\forall \rrho \in \{\rho,1\}^n \times \{1\}^m, 
\,\,\, \left| \dfrac{1}{j!}\pa_i^j
\right|_{\rrho} 
\eta^{j} \to 0 \,\,\,\, (j \in \N, |j| \to \infty) 
\end{equation}
(where $|\cdot|_{\rrho,E}$ denotes the operator norm induced by 
the previous  $|\cdot|_{\rrho,E}$). \par 
Then we see easily (by elementary calculus) that 
the overconvergence of $(E,\nabla)$ is equivalent to the 
following assertion: 
For each $\eta \in (0,1)\cap\Gamma^*$, 
there exists $\rho_0 \in (\eta,1)\cap\Gamma^*$ such that for any 
$\rho \in [\rho_0,1]\cap\Gamma^*$ and 
$\rrho = (\rho_i)_i \in \{\rho,1\}^n \times \{1\}^m$, we have 
$$ \min_i\{p^{-1/(p-1)}|\pa_i|_{\rrho,E,\sp}^{-1}\} > \eta. $$
Moreover, it is easy to see that we can replace the above inequality by 
$$ IR(E,\rrho) = 
\min_i\{p^{-1/(p-1)}\rho_i^{-1}|\pa_i|_{\rrho,E,\sp}^{-1}\} > \eta. $$
Then, using the contunuity of 
the function $\rrho \mapsto IR(E,\rrho)$ 
(Proposition \ref{kedlayaxiaoprop}), we can conclude that the overconvergence 
condition for $(E,\nabla)$ 
is equivalent to the condition $IR(E,\1)=1$. 
\end{proof}

Before the proof of the main theorem, finally we recall a technical 
lemma, which was proved in \cite[2.8]{cc}. 

\begin{lem}\label{tech}
Let $\cX$ be a smooth $p$-adic formal scheme over $\Spf O_K$, 
let $I \subseteq (0,1)$ be a closed aligned interval of positive length 
and let $a = \sum_{n \in \Z} a_nt^n$ be a non-zero element of 
$\Gamma(\cX_K \times A^1_K(I),\cO)$. Then there exist 
an open dense sub affine formal scheme $\cU \subseteq \cX$ and 
a closed aligned subinterval $I' \subseteq I$ of positive length 
satisfying the following conditions$:$ 
\begin{enumerate}
\item 
$a \in \Gamma(\cU_K \times A^1_K(I'),\cO^{\times})$. 
\item 
For any $u \in \cU_K$ and $\rho \in I'$, we have 
$|a(u)|_{\rho} = |a|_{\rho}$, where 
$a(u) := \sum_{n\in \Z} a_n(u)t^n \allowbreak 
\in \Gamma(u \times A^1_K(I'),\cO)$ and 
$|\cdot|_{\rho}$ denotes the $\rho$-Gauss norm. 
\end{enumerate}
\end{lem}

\begin{proof}
For the convenience of the reader, we reproduce the proof here. 
In this proof, $|\cdot|$ denotes the supremum norm. Let us write 
$I = [\alpha,\beta]$. By \cite[3.1.7, 3.1.8]{kedlayaI}, we have 
$$|a| = \max ( \sup_n(|a_n|\alpha^n), \sup_n(|a_n|\beta^n)) = 
\max ( \sup_{n\leq 0}(|a_n|\alpha^n), \sup_{n\geq 0}(|a_n|\beta^n)).$$ 
Let us define finite subsets 
$A \subseteq \Z_{\leq 0}$, $B \subseteq \Z_{\geq 0}$ by 
$A := \{n \leq 0\,\vert\, |a_n|\alpha^n = |a|\}, 
B := \{n \geq 0\,\vert\, |a_n|\beta^n = |a|\}.$ Then we have 
$A \cup B \not= \emptyset$. \par 
Let us first consider the case $A \not= \emptyset$. Let $n_0$ be 
the maximal element of $A$. Then, since $a_{n_0} \not= 0$, there 
exists an element $b \in K^{\times}$ such that 
$ba_{n_0} \in \Gamma(\cX,\cO_{\cX})$ and that the image $\ol{ba_{n_0}}$ of 
$ba_{n_0}$ in $\Gamma(X,\cO_X)$ is non-zero. Let $\cU \subseteq \cX$ be the 
open dense affine sub formal scheme such that $\ol{ba_{n_0}}$ is invertible on 
$\cU \times_{\cX} X$. Then we have $ba_{n_0} \in 
\Gamma(\cU,\cO_{\cU}^{\times})$. So, for all $u \in\cU_K$, we have 
$|a_{n_0}(u)| = |b^{-1}|$ and hence $|a_{n_0}(u)| = |a_{n_0}|$. 
(Here note that, for elements in $\Gamma(\cX_K, \cO)$, 
its supremum norm on $\cX_K$ is the same as that on $\cU_K$.) 
Next we prove the following claim: \\
\quad \\
{\bf claim.} \,\, There exists a closed aligned subinterval 
$I' \subseteq I$ of positive length 
such that $|a_n|\rho^n < |a_{n_0}|\rho^{n_0}$ 
for any $n \in \Z, \not= n_0$ and $\rho \in I'$. \\
\quad \\
Let us put $C := \{n \in \Z \,\vert\, \max(|a_n|\alpha^n, |a_n|\beta^n) 
\geq |a_{n_0}|\beta^{n_0}\}$. Then $C$ is a finite set containing $A$. 
If $n \in A, \not=n_0$, we have $|a_n|\alpha^n = |a_{n_0}|\alpha^{n_0}$ 
and $n < n_0 \leq 0$. Hence we have 
$|a_n|\rho^n < |a_{n_0}|\rho^{n_0}$ for any $\rho \in (\alpha,\beta]$. 
For $n \in C - A$, we have $|a_n|\alpha^n < |a_{n_0}|\alpha^{n_0}$. 
So there exists $\beta' \in (\alpha, \beta]$ such that, for any $n 
\in C-A$ and for any 
$\rho \in [\alpha,\beta']$, we have $|a_n|\rho^n < |a_{n_0}|\rho^{n_0}$. 
For $n \not\in C$, we have, for any $\rho \in I$, the inequalities 
$$ 
|a_n|\rho^n \leq \max(|a_n|\alpha^n,|a_n|\beta^n) < 
|a_{n_0}|\beta^{n_0} \leq |a_{n_0}|\rho^{n_0}. $$
Summing up these, we see the claim. \par 
Let us put $f := \sum_{n \not= n_0}(a_n/a_{n_0})t^{n-n_0} \in 
\Gamma(\cU_K \times A^1_K(I'),\cO)$ and 
take any $u\in \cU_K$, 
$\rho \in I'$. Then we have 
$$ |f(u)|_{\rho} \leq 
\dfrac{\sup_{n\not=n_0}(|a_n|\rho^n)}{|a_{n_0}(u)|\rho^{n_0}} = 
\dfrac{\sup_{n\not=n_0}(|a_n|\rho^n)}{|a_{n_0}|\rho^{n_0}} <1. $$
So we have $|f| <1$. So we have 
$a = a_{n_0}t^{n_0}(1+f) \in \Gamma(\cU_K \times A^1_K(I'),\cO^{\times}).$ 
Moreover, for any $u\in \cU_K$ and $\rho \in I'$, 
we have 
$$ 
|a(u)|_{\rho} = \sup_{n}(|a_n(u)|\rho^n) = |a_{n_0}(u)|\rho^{n_0} 
= |a_{n_0}|\rho^{n_0}
= \sup_{n}(|a_n|\rho^n) = |a|_{\rho}. $$
We can prove the lemma in the case $B\not= \emptyset$ in the same way. 
(In this case, we define $n_0$ to be the minimal element of $B$.) 
So we are done. 
\end{proof}

Now we are ready to prove Theorem \ref{mainthm}. 
The proof is similar to that of \cite[2.5, 2.9]{cc}. 

\begin{proof}[Proof of Theorem \ref{mainthm}]
As we explanined in the introduction, it suffices to prove the implication 
(2)$\,\Longrightarrow\,$(1) and by Proposition \ref{firstred}, it 
suffices to prove it for a formal smooth pair of the form 
$(\fG_{m,O_K}^n \times \fAf_{O_K}^m, \fAf_{O_K}^{n+m}) =: 
(\cX_0,\ol{\cX}_0)$. (Then we have 
$\cX_{0,K} = A^n_K[1,1] \times A^m_K[0,1], \ol{\cX}_{0,K} = 
A^{n+m}_K[0,1]$.) For $\lam \in (0,1]\cap\Gamma^*$, let 
$\fU_{0,\lam} = A^n_K[\lam,1] \times A^m_K[0,1]$ be as before. \par 
Let $(E,\nabla)$ be a $\nabla$-module on a strict neighborhood of 
$\cX_{0,K}$ in $\ol{\cX}_{0,K}$ satisfying the condition (2). Then 
it is defined on $\fU_{0,\lam}$ for some $\lam$. Let us 
assume that it is not overconvergent. Then, by Proposition \ref{oc-ir}, 
we have $IR(E,\1) <1$. So we have $p^{-1/(p-1)}|\pa_i|^{-1}_{\1,E,\sp} 
< 1$ for some $i \,(1 \leq i \leq n+m)$. 
Fix such index $i$. \par 
In the following, for $\rho \in [0,1]$, we denote 
$(1,...,1,\os{i}{\check{\rho}},1,...,1) \in [0,1]^{n+m}$ simply by 
$\rrho$. Let us denote the restriction of $E$ to 
\begin{equation*}
\begin{aligned}
(\fG^{n+m-1}_{m,O_K})_K \times A^1_K[\lam,1] & = 
A_K^{n+m-1}[1,1] \times A^1_K[\lam,1] \\ 
& \cong A_K^{i-1}[1,1] \times A^1_K[\lam,1] 
\times A_K^{n+m-i}[1,1] \,\,\,\, (\text{if $1 \leq i \leq n$}), \\
(\fG^{n+m-1}_{m,O_K})_K \times A^1_K[0,1] & = 
A_K^{n+m-1}[1,1] \times A^1_K[0,1] \\ 
& \cong A_K^{i-1}[1,1] \times A^1_K[0,1] 
\times A_K^{n+m-i}[1,1] \,\,\,\, (\text{if $n+1 \leq i \leq n+m$})
\end{aligned}
\end{equation*}
(where the second isomorphism is the permutation of the $i$-th factor and 
the last factor) also by $E$. \par 
Let us take a closed aligned intervals $I_r = [\alpha_r,\beta_r] 
\subseteq [\lam,1) \, (r \in \N)$ with $\alpha_r < \beta_r < 
\alpha_{r+1}\,(\forall r)$, $\lim_{r\to\infty}\alpha_r=1$ 
and put $A_r := \Gamma((\fG^{n+m-1}_{m,O_K})_K \times A_K(I),\cO)$, $\E_r := 
\Gamma((\fG^{n+m-1}_{m,O_K})_K \times A_K(I_r),E)$. 
Then $A_r$ is an integral domain and 
$\E_r$ is a finitely generated $A$-module. Let 
$\e_r := (\e_{r,1},...,\e_{r,\mu})$ be a basis of $\Frac A_r \otimes_{A_r} 
\E_r$ as $\Frac A_r$-vector space ($\mu = \rk E$) 
and let $(\f_{r,1}, ..., \f_{r,\nu_r})$ 
be a set 
of generator of $\E_r$ as $A_r$-module. Then there exist 
$b_{r,ij} := b'_{r,ij}/b''_{r,ij}, c_{r,ji} := c'_{r,ji}/c''_{r,ji} 
\in \Frac A_r 
\,(1 \leq i \leq \mu, 1 \leq j \leq \nu_r)$ such that 
$\e_{r,i} = \sum_{j=1}^{\nu_r} b_{r,ij} \f_{r,j} \,(\forall i), 
\f_{r,j} = \sum_{i=1}^{\mu} c_{r,ji} \e_{r,i} \,(\forall j)$. By 
Lemma \ref{tech}, there exists an open dense sub affine formal scheme 
$\cV_r \subseteq (\fG^{n+m-1}_{m,O_K})_K$ 
and closed aligned subinterval $I'_r \subseteq I_r$ 
of positive length 
such that $b''_{r,ij}, c''_{r,ji} \in \Gamma(\cV_K \times A^1_K(I'),
\cO^{\times})$. Then we see that $\e_r$ forms a basis of 
$\Gamma(\cV_{r,K} \times A^1_K(I'_r),E)$ as 
$\Gamma(\cV_{r,K} \times A^1_K(I'_r),\cO)$-module. 
For $s \in \N$, let $G_{r,s} \in 
\Mat_{\mu}(\Gamma(\cV_{r,K} \times A^1_K(I'_r),\cO))$ be 
the matrix expression of $\pa^s$ with respect to the basis $\e_r$ and 
let us put $V_r := \cV_r \times_{\cZ} Z$. 
Then, again by Lemma \ref{tech}, there exists a decreasing 
sequence $\{V_{r,s}\}_{s\in\N}$ of dense open subschemes in $V_r$ and 
a decreasing sequence $\{I'_{r,s}\}_{s\in \N}$ 
of closed aligned subintervals of $I'_r$ of positive length 
such that, for any 
separable closed point $x$ in $V_{r,s}$, any lift $\ti{x}$ of $x$ in 
$\fG^{n+m-1}_{m,O_K}$, any $s' \leq s$ and any $\rho_r \in I'_{r,s}$, 
we have the equality $|G_{r,s'}(\ti{x}_K)|_{\rho_r}=|G_{r,s'}|_{\rrho_r}$, 
where $G_{r,s'}(\ti{x}_K) \in \Mat_{\mu}(\Gamma(\ti{x}_{K} 
\times A^1_K(I'_r),\cO))$ is the pull-back of $G_{r,s'}$ by 
$\ti{x}_K \times A^1_K(I'_r) \hra \cV_{r,K} \times A^1_K(I'_r)$. 
Let $\pa_i(\ti{x}_K)$ be the action of $\pa_i$ on the pull-back 
$E(\ti{x}_K)$ of $E$ to $\ti{x}_K \times A^1_K[\lam,1]$ (if $1 \leq i \leq 
n$), $\ti{x}_K \times A^1_K[0,1]$ (if $n+1 \leq i \leq n+m$). 
Then, for any 
separable closed point $x$ in $\bigcap_s V_{r,s}$, 
any lift $\ti{x}$ of $x$ in $\fG^{n+m-1}_{m,O_K}$ and 
any $\rho_r \in \bigcap_{s\in\N}I'_{r,s}$, we have 
\begin{align}
p^{-1/(p-1)}\rho_r^{-1}|\pa_i(\ti{x}_K)|^{-1}_{\rho_r,E(\ti{x}_K),\sp} 
& = 
\min\{1,p^{-1/(p-1)}\rho_r^{-1}\varliminf_{s\to\infty}
|G_{r,s}(\ti{x}_K)|_{\rho_r}^{-1/s}\} \label{box1} \\ 
& = 
\min\{1,p^{-1/(p-1)}\rho_r^{-1}\varliminf_{s\to\infty}
|G_{r,s}|_{\rrho_r}^{-1/s}\} \nonumber \\ 
& = p^{-1/(p-1)}\rho_r^{-1}|\pa_i|^{-1}_{\rrho_r,E,\sp}. 
\nonumber 
\end{align}
Now let us take a separable point $x$ in 
$\bigcap_{r,s}V_{r,s}$ (which is possible since $k$ is uncountable) and 
its lift $\ti{x}$ in $\fG^{n+m-1}_{m,O_K}$. Then we have the equality 
\eqref{box1} for any $r \in \N$. Since the functions 
$$ \rho \mapsto p^{-1/(p-1)}\rho^{-1}|
\pa_i(\ti{x}_K)|^{-1}_{\rho,E(\ti{x}_K),\sp}, \,\,\, 
\rho \mapsto 
p^{-1/(p-1)}\rho^{-1}|\pa_i|^{-1}_{\rrho_r,E,\sp}. 
$$ 
are continuous by Proposition \ref{kedlayaxiaoprop} and 
Corollary \ref{kedlayaxiaocor}, this implies 
\begin{equation}\label{box2}
IR(E(\ti{x}_K),1) = p^{-1/(p-1)}|\pa_i|^{-1}_{\1,E,\sp} < 1.
\end{equation}
Now let us define the strict closed immersion 
$i:(\cY,\ol{\cY}) \hra (\cX_0,\ol{\cX}_0)$ of formal smooth pairs by 
$\ol{\cY} := \ti{x} \times \fAf^1_{O_K}$, 
\begin{align*}
i:\ol{\cY} & = \ti{x} \times \fAf^1_{O_K} \hra 
\fG^{n+m-1}_{O_K} \times \fAf^1_{O_K} \\ & \cong 
\fG^{i-1}_{O_K} \times \fAf^1_{O_K} \times \fG^{n+m-i}_{O_K} 
\hra \fAf^{n+m}_{O_K} = \ol{\cX}_0
\end{align*}
(where the isomorphism in the second line is the permutation of the $i$-th 
factor and the last factor) and 
\begin{equation*}
\cY := \ol{\cY} \times_{\ol{\cX}_0} \cX_0 = \left\{
\begin{aligned}
& \ti{x} \times \fG_{m,O_K}, \,\,\, \text{if $1 \leq i \leq n$}, \\
& \ti{x} \times \fAf^1_{O_K}, \,\,\, \text{if $n+1 \leq i \leq n+m$}. 
\end{aligned} \right.
\end{equation*}
Then $E(\ti{x}_K)$ is nothing but the pull-back $i^*E$ of $E$ by $i$. 
Hence, by \eqref{box2} and Proposition \ref{oc-ir}, 
$i^*(E,\nabla)$ is not overconvergent and this contradicts the assumption 
on $(E,\nabla)$. Hence $(E,\nabla)$ is overconvergent and so the proof 
of the theorem is finished. 
\end{proof}

\begin{rem}\label{refined}
Let $(\cX,\ol{\cX})$ be a formal smooth pair admitting a strict 
finite etale morphism 
$f: (\cX,\ol{\cX}) \lra (\cX_0, \ol{\cX}_0) := 
(\fG^n_{m,O_K} \times \fAf^m_{O_K}, \fAf^{n+m}_{O_K})$. 
For a closed sub formal scheme $\ti{x} \hra \fG^{n+m-1}_{m,O_K}$ which is 
etale over $\Spf O_K$ and $1 \leq i \leq n+m$, let $\ol{\cY}_{\ti{x},i,0} 
\hra \ol{\cX}_0$ be the closed immersion defined by 
\begin{align*}
\ol{\cY}_{\ti{x},i,0} & := \ti{x} \times \fAf^1_{O_K} \hra 
\fG^{n+m-1}_{O_K} \times \fAf^1_{O_K} \\ & \cong 
\fG^{i-1}_{O_K} \times \fAf^1_{O_K} \times \fG^{n+m-i}_{O_K} 
\hra \fAf^{n+m}_{O_K} = \ol{\cX}_0
\end{align*}
(where the isomorphism in the second line is the permutation of the $i$-th 
factor and the last factor), let $\cY_{\ti{x},i,0} := \ol{\cY}_{\ti{x},i,0} 
\times_{\ol{\cX}_0} \cX_0$ and let 
$(\cY_{\ti{x},0}, \ol{\cY}_{\ti{x},0})$ be the pull-back of 
$(\cY_{\ti{x},0}, \ol{\cY}_{\ti{x},0})$ by $f$. Then, by the proofs of 
Proposition \ref{firstred} and Theorem \ref{mainthm} given above, we see 
the following refined version of Theorem \ref{mainthm} in this case: 
For $(E,\nabla) \in \MIC((\cX_K,\ol{\cX}_K)/K)$, the following are 
equivalent: 
\begin{enumerate}
\item $(E,\nabla)$ is overconvergent. 
\item For any $\ti{x}$ as above and for any $1 \leq i \leq n+m$, 
the restriction of $(E,\nabla)$ to $\MIC((\cY_{\ti{x},i,K}, 
\ol{\cY}_{\ti{x},i,K})/K)$ is overconvergent. 
\end{enumerate}
\end{rem}

\section{Proof of the second main theorem}

In this section, we give a proof of the second main theorem 
(Theorem \ref{mainthm2}) in this paper. To do so, first we prove a 
partial generalization of a result of Kedlaya \cite[Theorem 2]{kedlayamore} 
on etale covers of smooth $k$-varieties to the case of smooth formal schemes 
and smooth schemes over $O_K$. 
Throughout this section, let $\pi$ be a fixed uniformizer of $O_K$ and 
for a scheme 
$\X$ over $\Spec O_K$, we denote the generic fiber of it by $\X_K$, 
the special fiber of it by $\X_k$ and the $p$-adic completion of it by 
$\wh{\X}$. For a $p$-adic formal scheme $\cX$ over $\Spf O_K$, we denote 
the special fiber of it by $\cX_k$. \par 
Let $(\cX,\ol{\cX})$ be a formal smooth pair and let $x$ be a 
closed point of $\cX_k$. Let us put 
$\cD := \ol{\cX} - \cX = \bigcup_{i=1}^q \cD_i$ 
(where $\cD_i$'s are connected and smooth over $\Spf O_K$) and define 
$I_{\ni x}, I_{\not\ni x}, \cD_{\ni x}, \cD_{\not\ni x}$ in the 
following way: 
\begin{align*}
& I_{\ni x} := \{i \,\vert\, x \in \cD_i\}, \qquad 
I_{\not\ni x} :=  \{i \,\vert\, x \not\in \cD_i\}, \\
& \cD_{\ni x} := \bigcup_{i\in I_{\ni x}} \cD_i, \qquad 
 \cD_{\not\ni x} := \bigcup_{i\in I_{\not\ni x}} \cD_i. 
\end{align*}
We say that the formal smooth pair $(\cX,\ol{\cX})$ satisfies the condition 
$(*)_x$ if, for any $I \subseteq I_{\ni x}$, 
$\bigcap_{i \in I}\cD_i$ is irreducible. We will prove the following 
result, which is a formal scheme version of \cite[Theorem 2]{kedlayamore}: 

\begin{thm}\label{moreformal}
Let $(\cX,\ol{\cX})$ be a formal smooth pair with $\dim \ol{\cX} = n$ and 
$\ol{\cX}$ projective over $\Spf O_K$. Let $x$ be a closed point 
in $\ol{\cX}_k$ and assume that $(\cX,\ol{\cX})$ satisfies the condition 
$(*)_x$. Let us put $\cD := \ol{\cX} - \cX$, let us define 
$\cD_{\ni x}, \cD_{\not\ni x}$ as above and let $\cD_{\ni x} = 
\bigcup_{i=1}^r \cD_i$ be the decomposition of $\cD_{\ni x}$ into 
irreducible components. Then there exists a finite flat morphism 
$f: \ol{\cX} \lra \fP^n_{O_K}$ satisfying the following conditions$:$ 
\begin{enumerate}
\item 
$f$ is etale on $f^{-1}(\fAf^n_{O_K})$. 
\item 
We have $\cD_{\not\ni x} \subseteq f^{-1}(H_0)$ and 
$x \in f^{-1}(\fAf^n_{O_K})$, where $H_0 := \fP^n_{O_K} - \fAf^n_{O_K} 
\subseteq \fP^n_{O_K}$ is the hyperplane at infinity. 
\item 
For $1 \leq i \leq r$, we have $\cD_i \subseteq f^{-1}(H_i)$, where 
$H_i \subseteq \fP^n_{O_K}$ is the $i$-th coordinate hyperplane. 
$($Hence, by combining with $(1)$, we see that 
$\cD_i \cap f^{-1}(\fAf^n_{O_K}) \subseteq f^{-1}(H_i \cap \fAf^n_{O_K})$ 
is an open and closed immersion.$)$
\end{enumerate}
\end{thm}

We prove Theorem \ref{moreformal} by mimicking the proof of 
\cite[Theorem 2]{kedlayamore}. First we prove several premilinary lemmas. 

\begin{lem}[{cf. \cite[Lemma 3]{kedlayamore}}]\label{lemma3}
Let $\cX$ be a $p$-adic formal scheme projective over $\Spf O_K$ and 
let $\cL$ be an ample line bundle on $\cX$. Let $\alpha_0, ..., \alpha_n$ 
be sections of $\cL$ such that the intersection of zero loci of $\alpha_i$ 
$(0 \leq i \leq n)$ is empty. Then $\alpha_i$'s induce a finite morphism 
$\alpha: \cX \lra \fP^n_{O_K}$. 
\end{lem}

\begin{proof}
Since $\alpha$ is proper, it suffices to check that $\alpha$ modulo 
$\pi O_K$ is quasi-finite and it is proved in 
\cite[Lemma 3]{kedlayamore}. 
\end{proof} 

\begin{lem}[{cf. \cite[Lemma 4]{kedlayamore}}]\label{lemma4}
Let $\cX$ be a $p$-adic formal scheme projective and flat over $\Spf O_K$ and 
let $\cL$ be an ample line bundle on $\cX$. Let $\cD \subseteq \cX$ be a closed
 sub formal scheme and let $\cZ \subseteq \cX$ be a $0$-dimensional 
closed sub formal scheme such that $\cD \cap \cZ$ is empty. 
Then, for $l$ sufficiently divisible, there exists a section of 
$\cL^{\times l}$ 
which vanishes on $\cD$ but not on any point in $\cZ_k$. 
\end{lem}

\begin{proof}
Let $\X, {\bold L}, \D$ be the algebraizations of $\cX, \cL, \cD$, 
respectively. Then, 
for sufficiently large $a>0$, the map 
$\Gamma(\X, {\bold L}^{\otimes a}) = 
\Gamma(\cX, \cL^{\otimes a}) \lra \Gamma(\cX_k,\cL^{\otimes a})$ 
is surjective. By this fact and the proof of \cite[Lemma 4]{kedlayamore}, 
we see that, for $a$ sufficiently large, 
there exists a section $s_0$ of ${\bold L}^{\otimes a}$ 
which does not vanish on any point in $\cZ_k$. On the complement of the 
zero locus of $s_0$ in $\X$, 
there exists a regular function vanishing on $\D$, 
not on any point in $\cZ_k$ and this function has the form $s_1s_0^{-b}$ 
for some $b$. Then, for any $l$ divisible by $ab$, the image of the section 
$s_1^{l/ab} \in \Gamma(\X, {\bold L}^{\otimes l})$ in 
$\Gamma(\cX, \cL^{\otimes l})$ has the desired property. 
\end{proof}

\begin{lem}[{cf. \cite[Lemma 5]{kedlayamore}}]\label{lemma5}
Let $\cX$ be a $p$-adic formal scheme projective and flat over $\Spf O_K$ 
with $\dim \cX = n$ and 
let $\cL$ be an ample line bundle on $\cX$. Let $\cD \subseteq \cX$ be a closed
 sub formal scheme and let $\cZ \subseteq \cX$ be a $0$-dimensional 
closed sub formal scheme such that $\cD \cap \cZ$ is empty. Let $m$ be 
an integer with $0 \leq m \leq n$ and let $\cD_1,...,\cD_m$ be divisors 
of $\cX$ such that, for any $T \subseteq [1,m]$, 
$\cD \cap \bigcap_{t \in T}\cD_t$ has codimension at least $|T|$ 
in $\cD$. Then, for $l$ sufficiently divisible, there exist 
sections $s_1,...,s_n$ of $\cL^{\otimes l}$ with no common zero on 
$\cD_k$ such that $s_i$ vanishes on $\cZ \, (1 \leq i \leq n)$ 
and that $s_i$ vanishes on $\cD_i$ for $1 \leq i \leq m$. 
\end{lem} 

\begin{proof} 
Although the following proof is the same as that in 
\cite[Lemma 5]{kedlayamore}, we reproduce the proof here for the reader's 
convenience. We construct the desired section inductively. Suppose that, for $0 \leq j < n$, there exist positive integers 
$l_i$ and sections $s'_i$ of $\cL^{\otimes l_i}$ \,$(1 \leq i \leq j)$ 
satisfying the following conditions: 
\begin{enumerate}
\item[(a)] Each $s'_i$ vanishes on $\cZ$. 
\item[(b)] For $1\leq i \leq \min (j,m)$, $s'_i$ vanishes on $\cD_i$. 
\item[(c)] For any $T \subseteq [j+1,m]$, 
$Y_{j,T} := \cD \cap \bigcap_{t\in T}\cD_t \cap 
\bigcap_{i=1}^j \{\text{zero locus of $s'_i$}\}$ 
has codimension $\geq j + |T|$ in $\cD$.  
\end{enumerate}
(This is true for $j=0$ by hypothesis.) Then let $\cZ_j$ be a 
$0$-dimensional closed sub formal scheme of $\cD - \cD_{j+1}$ 
which meets each irreducible component of $Y_{j,T}$ having 
codimension $j + |T|$ in $\cD$, for each $T \subseteq [j+2,m]$. 
(By (c), none of these components are contained in $\cD \cap \cD_{j+1}$. 
Hence such $\cZ_j$ actually exists.) By Lemma \ref{lemma4}, 
for $l_{j+1}$ sufficiently divisible, there exists a section $s'_{j+1}$ 
of $\cL^{\otimes l_{j+1}}$ vanishing along $\cZ \cup \cD_{j+1}$, not on 
$(\cZ_j)_k$. Then, with this $s'_{j+1}$, the conditions (a), (b), (c) are 
satisfied with $j$ replaced by $j+1$. \par 
So we can take sections $s'_1,...,s'_n$ such that the conditions 
(a), (b), (c) are satisfied with $j=n$. Let $l_0$ be the least common 
multiple of $l_j$'s. Then for any $l$ divisible by $l_0$, 
the sections $s_i := {s'_i}^{l/l_i}\,(1 \leq i \leq n)$ of $\cL^{\otimes l}$ 
satisfy the desired properties. 
\end{proof}

\begin{lem}[{cf. \cite[Lemma 6]{kedlayamore}}]\label{lemma6}
Let $\cX$ be a $p$-adic formal scheme projective and smooth over $\Spf O_K$ 
with $\dim \cX = n$ and 
let $\cL$ be an ample line bundle on $\cX$. Let 
$\bigcup_{i=1}^m \cD_i$ be a relative simple normal crossing divisor 
in $\cX$ $($each $\cD_i$ being smooth$)$ and let $\ti{x}$ be a closed 
sub formal scheme of $\bigcap_{i=1}^m \cD_i$ connected and etale over 
$\Spf O_K$. Let $\alpha$ be a section of $\cL$ whose zero locus 
$\cD$ does not meet $\ti{x}$. Then, for $l$ sufficiently divisible, 
there exist sections $\delta_1,...,\delta_n$ of $\cL^{\otimes l}$ 
satisfying the following conditions$:$ 
\begin{enumerate}
\item 
$\delta_i$'s have no common zero on $\cD$, that is, 
$\bigcap_{i=1}^n \{\text{zero locus of $\delta_i$}\} \cap \cD$ is 
empty. Hence $\alpha^l, \delta_1,...,\delta_n$ define a finite 
morphism $f: \cX \lra \fP^n_{O_K}$ by Lemma \ref{lemma3}. 
\item 
$f$ is unramified at $\ti{x}_k$. 
\item 
For $1 \leq i \leq m$, $\delta_i$ vanishes on $\cD_i$. 
\end{enumerate}
\end{lem} 

\begin{proof} 
Let $\X, \bigcup_{i=1}^m \D_i, {\bold L}, \D$ be the algebraization of 
$\cX, \bigcup_{i=1}^m \cD_i, \cL, \cD$. Then $\alpha$ induces 
a section $\ti{\alpha}$ 
of ${\bold L}$ with zero locus $\D$. Hence ${\bold U} := \X - \D$ 
is an affine scheme. Then we can find 
$\ti{f}_1,..., \ti{f}_n \in \Gamma({\bold U}, \cO)$ such that 
$\ti{f}_i$ vanishes on $\D_i$ for $1 \leq i \leq m$ 
and that they induces a morphism 
${\bold U} \lra \Af^n_{O_K}$ which is unramified at 
$\ti{x}_k \in \cX_k = \X_k \subseteq \X$. 
Write $\ti{f}_i = \ti{\beta}_i\ti{\alpha}^{-j_i}$ for some positive integer 
$j_i$ and $\beta_i \in \Gamma(\X,{\bold L}^{\otimes j_i})$. \par 
Put $\cU := \cX - \cD$ and denote the image of $\ti{f}_i$ 
(resp. $\ti{\beta}_i$) in 
$\Gamma(\cU,\cO)$ (resp. $\Gamma(\cX, \cL^{\otimes j_i})$ 
by $f_i$ (resp. $\beta_i$). Then $f_i$'s induce a map 
$\cU \lra \fAf^n_{O_K}$ which is unramified at $\ti{x}_k$. 
By Lemma \ref{lemma5}, 
for $l$ sufficiently divisible, we can choose $\gamma_1,...,\gamma_n \in 
\Gamma(\cX,\cL^{\otimes l})$ such that $\gamma_i\,(1 \leq i \leq n)$ vanishes 
on $\ti{x}$, $\gamma_i$ vanishes on $\cD_i$ for $1 \leq i \leq m$ and that 
$\gamma_1,...,\gamma_m$ has no common zero on $\cD$. Now let us put 
$\delta_i := \beta_i \alpha^{2l-j_i} + \gamma_i^2 \in \Gamma(\cX,\cL^{\otimes 
2l})$ for $1 \leq i \leq n$. 
Then we can check that this $\delta_i$'s satisfy the required properties. 
\end{proof}

Now we give a proof of Theorem \ref{moreformal}: 

\begin{proof}[Proof of Theorem \ref{moreformal}]
The proof is essentially the same as that of \cite[Theorem 2]{kedlayamore}. 
Let $\cL$ be an ample line bundle on $\ol{\cX}$ and let $\ti{x}$ be 
a lift of $x$ in $\cD_{\ni x}$. Since we may replace $\cL$ by its 
tensor power, we may assume by Lemma \ref{lemma4} that there exists a 
section $s$ of $\cL$ which vanishes on $\cD_{\not\ni x}$ and not on 
$x$. By Lemma \ref{lemma6}, there exist a positive integer $m$ and 
sections $s_1,...,s_n$ of $\cL^{\otimes m}$ satisfying the following 
conditions: 
\begin{enumerate}
\item[(a)] 
$s_i$'s have no common zero on $\cD_{\not\ni x}$ and so 
$s^m, s_1, ..., s_n$ defines a finite morphism 
$g: \ol{\cX} \lra \fP^n_{O_K}$ by Lemma \ref{lemma3}. 
\item[(b)] 
$g$ is unramified at $x$. 
\item[(c)] 
$s_i$ vanishes on $\cD_i$ for $1 \leq i \leq r$. 
\end{enumerate}
Let $\cD$ be the zero locus of $s$. (Then we have $\cD_{\not\ni x} 
\subseteq \cD$ and $\cD$ does not contain $x$.) The locus on 
$\ol{\cX} - \cD$ where $g$ is unramified is open. Let $\cE$ be 
its complement in $\cX$ and endow $\cE$ a structure of a closed 
sub formal scheme of $\cX$ satisfying $\cD \subseteq \cE$. 
(Note that we have $x \notin \cE$.) 
Then, by Lemma \ref{lemma4}, there exists some $q$ and $t \in 
\Gamma(\ol{\cX}, \cL^{\otimes qm})$ vanising on $\cE$ and not on $x$. 
Let $\cZ$ be the vanishing locus of $t$ (so $\cE \subseteq \cZ, 
x \notin \cZ$ and $\cZ$ is pure codimension $1$ in $\ol{\cX}$). 
For any $T \subseteq [1,r]$, $\bigcap_{t \in T}\cD_t$ is irreducible 
by the assumption $(*)_x$ and it contains $x$. Hence 
$\cZ \cap \bigcap_{t \in T}\cD_t$ has codimension $|T|$ in $\cZ$. 
So we can apply Lemma \ref{lemma5} and we can conclude that 
there exist some $l$ and sections $t_1,..., t_n \in 
\Gamma(\ol{\cX}, \cL^{\otimes lqm})$ with no common zero on $\cZ$ such 
that $t_i$ vanishes on $\cD_i$ for $1 \leq i \leq r$. Now let us put 
\begin{align*}
u_0 & := t^{pl}, \\
u_i & := s_is^{m(pq-1)}t^{p(l-1)} + t_i^p \,\,\,\,(1 \leq i \leq n). 
\end{align*}
Then they defines a finite morphism $f: \ol{\cX} \lra \fP^n_{O_K}$ by 
Lemma \ref{lemma3}. \par 
For $y \in (\cX-\cZ)_k$, $g$ is unramified at $y$ by assumption, that is, 
$d(s_i/s^m)$'s are linearly independent at $y$. On the other hand, 
we see that $d(u_i/u_0)$ has the form $(s^{qm}/t)^p d(s_i/s^m)$ modulo $p$. 
So $d(u_i/u_0)$'s are also linearly independent at $y$. Hence $f$ is 
unramified at $y$. It is easy to check that $f$ satisfies the 
conditions (2) and (3) in the statement of the theorem. Hence, to prove the 
theorem, it suffices to prove the flatness of $f$, that is, the 
flatness of $f$ modulo $\pi^m O_K$ for all $m \in \N$. \par 
Note that $f$ modulo $\pi O_K$ is flat, because $\ol{\cX}_k, \P^n_k$ are 
regular of dimension $n$ and all the fibers of $f$ modulo $\pi O_K$ is 
$0$-dimensional (see \cite[6.1]{ega}). So it suffices to prove the 
following claim to finish the proof of the theorem: \\
\quad \\
{\bf claim.} \,\,\, 
Let $f_m: \cP_m \lra \cQ_m$ be a finite morphism of flat 
$O_K/\pi^mO_K$-schemes such that $f_m$ modulo $\pi O_K$ is flat. 
Then $f_m$ is flat. \\
\quad \\
For $\nu \leq m$, let $f_{\nu}: \cP_{\nu} \lra \cQ_{\nu}$ be 
$f_m$ modulo $\pi^{\nu} O_K$. To prove the claim, we may assume that 
$f_{\nu}$ is flat for $\nu \leq m-1$ as induction hypothesis. 
Moreover, since the assertion is local on $\cQ_m$, we may assume that 
$\cQ_m = \Spec A_m$ is the spectrum of a local ring. Let us put 
$\cQ_{\nu} := \Spec A_{\nu}$. Then $\cP_{\nu}$, being finite over $\cQ_{\nu}$ 
is also affine (so put $\cP_{\nu} =: \Spec B_{\nu}$) and 
since $\cP_{m-1}$ is finite flat over $\cQ_{m-1}$, there exists some 
$\mu \in \N$ and an isomorphism $\varphi_{m-1}: A_{m-1}^{\oplus \mu} 
\os{=}{\lra} B_{m-1}$ as $A_{m-1}$-modules. Let 
$\varphi_{\nu}: A_{\nu}^{\oplus \mu} \os{=}{\lra} B_{\nu}$ be $\varphi_m$ 
modulo $\nu\,(\nu \leq m-1)$ and let us take a homomorphism 
$\varphi_m: A_m^{\oplus \mu} \lra B_m$ as $A_m$-modules lifting 
$\varphi_{m-1}$. Then we have the following diagram whose horizontal 
lines are exact by flatness of $\cP_m, \cQ_m$ over $O_K/\pi^m O_K$: 
\begin{equation*}
\begin{CD}
0 @>>> A_{m-1}^{\oplus \mu} @>{\pi}>> A_m^{\oplus \mu} @>>> A_1^{\oplus \mu} 
@>>> 0 \\ 
@. @V{\|}V{\varphi_{m-1}}V @VV{\varphi_m}V @V{\|}V{\varphi_1}V \\
0 @>>> B_{m-1} @>{\pi}>> B_m @>>> B_1 @>>> 0. 
\end{CD}
\end{equation*}
From this diagram, we see that $\varphi_m$ is also an isomorphism and so 
$f_m$ is flat. Hence the claim is proved and so the proof of the theorem 
is finished. 
\end{proof}

Let $(\X,\ol{\X})$ be a smooth pair over $O_K$, let $x$ be a closed 
point on $\ol{\X}_k$ and let us put 
$\D := \ol{\X} - \X = \bigcup_{i=1}^r \D_i$ 
(where $\D_i$'s are connected and smooth over $\Spec O_K$). 
Then we can define $I_{\ni x}, I_{\not\ni x}, 
\D_{\ni x}, \D_{\not\ni x}$ as in the case of 
formal smooth pair. 
We say that the smooth pair $(\X,\ol{\X})$ over $O_K$ 
satisfies the condition 
$(*)_x$ if, for any $I \subseteq I_{\ni x}$, 
$\bigcap_{i \in I}\D_i$ is irreducible. With this terminology, 
we have the following algebraic variant of Theorem \ref{moreformal} as 
a corollary: 

\begin{cor}\label{morealg}
Let $(\X,\ol{\X})$ be a smooth pair over $O_K$ with 
$\dim (\ol{\X}/O_K) = n$ and 
$\ol{\X}$ projective over $\Spec O_K$. Let $x$ be a closed point 
in $\ol{\X}$ and assume that $(\X,\ol{\X})$ satisfies the condition 
$(*)_x$. Let us put $\D := \ol{\X} - \X$, let us define 
$\D_{\ni x}, \D_{\not\ni x}$ as above and let $\D_{\ni x} = 
\bigcup_{i=1}^r \D_i$ be the decomposition of $\D_{\ni x}$ into 
irreducible components. Then there exists a finite flat morphism 
$f: \ol{\X} \lra \P^n_{O_K}$ satisfying the following conditions$:$ 
\begin{enumerate}
\item 
$f$ is etale on a neighborhood of $f^{-1}(\Af^n_{k})$. 
\item 
We have $\D_{\not\ni x} \subseteq f^{-1}(H_0)$ and 
$x \in f^{-1}(\Af^n_{O_K})$, where $H_0 := 
\P^n_{O_K} - \Af^n_{O_K} \subseteq \P^n_{O_K}$ is the 
hyperplane at infinity. 
\item 
For $1 \leq i \leq r$, we have $\D_i \subseteq f^{-1}(H_i)$, where 
$H_i \subseteq \P^n_{O_K}$ is the $i$-th coordinate hyperplane. 
$($Hence we see that 
$\D_i \cap f^{-1}(V) \subseteq f^{-1}(H_i \cap V)$ 
is an open and closed immersion if $f^{-1}(V) \lra V$ is etale.$)$
\end{enumerate}
\end{cor}

\begin{proof} 
Since the formal smooth pair $(\wh{\X},\wh{\ol{\X}})$ satisfies the 
assumption of Theorem \ref{moreformal}, we have a finite flat morphism 
$\wh{f}: \wh{\ol{\X}} \lra \fP^n_{O_K}$ satisfying the conclusion of 
Theorem \ref{moreformal}. Note that the morphism $\wh{f}$ is induced by 
some sections $u_0,...,u_n \in \Gamma(\wh{\ol{\X}}, \cL)$ for some 
ample line bundle $\cL$ on $\wh{\ol{\X}}$. Let ${\bold L}$ be the 
algebraization of $\cL$, which is an ample line bundle on $\ol{\X}$. 
Then $u_0,...,u_n$ can be regarded as sections in 
$\Gamma(\ol{\X},{\bold L})$ and they do not have common zero. 
So they induce a morphism $f: \ol{\X} \lra \fP^n_{O_K}$, which is 
the algebraization of $\wh{f}$, and we can see that it is finite by the 
argument of \cite[Lemma 3]{kedlayamore}. Then $f_*\cO_{\ol{\X}}$ is 
a coherent sheaf on $\cO_{\P^n_{O_K}}$ such that 
$f_*\cO_{\ol{\X}} \otimes_{\cO_{\P^n_{O_K}}} \cO_{\fP^n_{O_K}} = 
\wh{f}_*\cO_{\wh{\ol{\X}}}$ is flat over $\cO_{\fP^n_{O_K}}$. Hence 
$f_*\cO_{\ol{\X}}$ is flat over $\cO_{\P^n_{O_K}}$ by GAGA. So $f$ is 
a finite flat morphism. We can check that this $f$ satisfies the required 
properties by using Theorem \ref{moreformal} and GAGA. 
\end{proof} 

We prove one more preliminary lemma before the proof of the second main 
theorem. Let $(\X,\ol{\X})$ be a smooth pair over $O_K$ and put 
$\D := \ol{\X} - \X = \bigcup_{i=1}^r \D_i$, where $\D_i$'s are connected 
and smooth over $O_K$. Then we call 
a closed subscheme ${\bold C} \hra \ol{\X}$ {\it toric} if, 
locally on $\ol{\X}$, 
${\bold C}$ has the form $\bigcap_{i \in I}\D_i$ for some 
$I \subseteq [1,r]$. A morphism $\varphi: \ol{\X}' \lra \ol{\X}$ is called 
{\it a toric blow-up} if it is a blow-up with toric center, and for 
$x \in \ol{\X}$, the morphism $\varphi$ is called 
{\it a toric blow-up away from} $x$ if it is a toric blow-up whose center 
does not contain $x$. If $\varphi: \ol{\X}' \lra \ol{\X}$ is a toric blow-up 
and if we put $\X' := \varphi^{-1}(\X)$, it induces the strict morphism of 
smooth pairs $(\X',\ol{\X}') \lra (\X,\ol{\X})$ over $O_K$, which we also 
denote by $\varphi$. Then we have the following lemma: 

\begin{lem}\label{lemma0}
Let $(\X,\ol{\X})$ be a smooth pair over $O_K$ and let $x$ be a closed 
point of $\ol{\X}_k$. Then there exists a morphism of smooth pairs 
$\varphi: (\X',\ol{\X}') \lra (\X,\ol{\X})$ over $O_K$ which is a compsoition 
of toric blow-ups away from $x$ such that $(\X',\ol{\X}')$ satisfies 
$(*)_x$. 
\end{lem}

\begin{proof} 
Let $\D = \bigcup_{i=1}^r \D_i$ and $I_{\ni x}$ as before. Then, 
for any $I \subseteq I_{\ni x}$, we have the decomposition 
$\bigcap_{i \in I}\D_i = \D_{I,0} \coprod \D_{I,1}$, where 
$\D_{I,0}$ is the connected component of $\bigcap_{i \in I}\D_i$ containing 
$x$ and $\D_{I,1}$ is the union of connected components not containing $x$. 
Then, by taking the composition of blow-ups along the (strong transforms of) 
$\D_{I,1}$'s, we obtain the desired morphism $\varphi$. 
\end{proof} 

Now we give a proof of Theorem \ref{mainthm2}. 

\begin{proof}[Proof of Theorem \ref{mainthm2}]
For a closed point $x$ of $\ol{\X}_k$, let us take a morphism of smooth pairs 
$\varphi: (\X',\ol{\X}') \lra (\X,\ol{\X})$ over $O_K$ which is a compsoition 
of toric blow-ups away from $x$ such that $(\X',\ol{\X}')$ satisfies 
$(*)_x$. (This is possible by Lemma \ref{lemma0}.) Let us put 
$\D' := \ol{\X}' - \X'$ and define 
$\D'_{\ni x} = \bigcup_{i=1}^r \D'_i, \D'_{\not\ni x} 
\subseteq \D'$ as before. Then we see that the exceptional divisor of 
$\varphi$ is contained in $\D'_{\not\ni x}$. Since $(\X',\ol{\X}')$ satisfies 
$(*)_x$, we can apply Corollary \ref{morealg} to it: Let 
$f: \ol{\X}' \lra \P^n_{O_K}$ be a morphism satisfying the conditions 
in the statement of Corollary \ref{morealg}, and let 
$V$ be an open subscheme of $\Af^n_{O_K}$ containing $\Af^n_k$ 
such that $f|_{f^{-1}(V)}: f^{-1}(V) \lra V$ is finite etale (which exists 
by (1) of Corollary \ref{morealg}) and let us put 
$\ol{\X}'_x := f^{-1}(V)$. Then, since we have 
$\D'_{\not\ni x} \cap \ol{\X}'_x = \emptyset$ by (2) of Corollary 
\ref{morealg}, $\varphi$ induces the isomorphism 
$\ol{\X}'_x \os{=}{\lra} \varphi(\ol{\X}'_x)$. \par 
For $1 \leq i \leq r$, let $H_i$ be the intersection of the $i$-th 
hyperplane in $\P^n_{O_K}$ with $V$. Then, by (3) of Corollary 
\ref{morealg}, $\D'\cap \ol{\X}'_x = \D'_{\ni x} \cap \ol{\X}'_x$ 
is a sub relative simple normal crossing divisor of 
$\bigcup_{i=1}^r f^{-1}(H_i)$ whose complementary divisor $C$ does not 
contain $x$. Now let us put 
$\X'_x := \X' \cap \ol{\X}'_x$, $\X''_x := \X'_x - C$, 
$\ol{\X}''_x := \ol{\X}'_x - C$. Then we have the diagram 
$$ 
(\X,\ol{\X}) \os{\varphi}{\lla} (\X'_x, \ol{\X}'_x) \os{j}{\lla} 
(\X''_x, \ol{\X}'_x) \os{f}{\lra} (V-\bigcup_{i=1}^r H_i, V) $$
(where $\varphi, f$ is the morphism induced by $\varphi: 
(\X',\ol{\X}') \lra (\X,\ol{\X})$, $f: \ol{\X}' \lra \P^n_{O_K}$ respectively 
and $j$ is the canonical inclusion) such that $\varphi \circ j$ induces 
the isomorphism $(\X''_x, \ol{\X}''_x) \os{=}{\lra} 
(\X \cap \varphi(\ol{\X}''_x), \varphi(\ol{\X}''_x))$. \par 
First we prove that Theorem \ref{mainthm2} is reduced to the following 
claim: \\
\quad \\
{\bf claim.} \,\,\, 
Let $(E,\nabla)$ be an object in $\MIC(\X'_{x,K}/K)$ and assume that, for any 
strict closed immersion $i:(\Y,\ol{\Y}) \hra (\X'_x, \ol{\X}'_x)$ with 
$\dim (\ol{\Y}/O_K) = 1$, $i^*(E,\nabla) \in \MIC(\Y_K/K)$ is overconvergent. 
Then the restriction of $(E,\nabla)$ by $j$ (which is an object in 
$\MIC(\X''_{x,K}/K)$) is overconvergent. \\
\quad \\
Indeed, assume the claim is true and assume we are given an object 
$(E,\nabla)$ in $\MIC(\X_K/K)$ satisfying the condition (2) of Theorem 
\ref{mainthm}. Then, since $\varphi: (\X'_x, \ol{\X}'_x) 
\allowbreak \lra (\X,\ol{\X})$ 
is a strict open immersion, the restriction of $(E,\nabla)$ to 
$\MIC(\X'_{x,K}/K)$ satisfies the assumption of the claim. Hence 
 the restriction of $(E,\nabla)$ by $j$ is overconvergent. So 
the restriction of $(E,\nabla)$ by $(\X''_x, \ol{\X}''_x) = 
(\X \cap \varphi(\ol{\X}''_x), \varphi(\ol{\X}''_x)) \hra (\X, \ol{\X})$ is 
also overconvergent. Since this is true for any $x$, we see that 
$(E,\nabla)$ is overconvergent by Proposition \ref{cov}. So we are reduced 
to proving the claim above. \par 
In the following, we prove the claim. (So let $(E,\nabla)$ be as 
in the claim.) Note that the $p$-adic completion 
of the diagram 
$$ (\X'_x, \ol{\X}'_x) \os{j}{\lla} (\X''_x, \ol{\X}'_x) \os{f}{\lra} 
(V - \bigcup_{i=1}^r H_i, V) $$ 
has the following form: 
$$ (\wh{\X}'_x, \wh{\ol{\X}}'_x) \os{\wh{j}}{\lla} 
(\wh{\X}''_x, \wh{\ol{\X}}'_x) \os{\wh{f}}{\lra} 
((V - \bigcup_{i=1}^r H_i)^{\wedge}, \wh{V}) = 
(\fG^r_{O_K} \times \fAf^{n-r}_{O_K}, \fAf^n_{O_K}). $$ 
For a closed subscheme $y \hra \G^{n-1}_{m,O_K}$ which is 
etale over $\Spec O_K$ and $1 \leq i \leq n$, let $\ol{\Y}_{y,i,0} 
\hra \ol{\cX}_0$ be the closed immersion defined by 
\begin{align*}
\ol{\Y}_{y,i,0} & := y \times \Af^1_{O_K} \hra 
\G^{n-1}_{O_K} \times \Af^1_{O_K} \\ & \cong 
\G^{i-1}_{O_K} \times \Af^1_{O_K} \times \G^{n-i}_{O_K} 
\hra \Af^{n}_{O_K} 
\end{align*}
(where the isomorphism in the second line is the permutation of the $i$-th 
factor and the last factor). Then let us put 
$$ 
\ol{\Y}_{y,i} := f^{-1}(\ol{\Y}_{y,i,0} \cap V) \supseteq 
\Y'_{y,i} := \ol{\Y}_{y,i} \cap \X'_x \supseteq 
\Y''_{y,i} := \ol{\Y}_{y,i} \cap \X''_x. 
$$ 
Then it is easy to see that $(\Y''_{y,i}, \ol{\Y}_{y,i})$ is a smooth 
pair over $O_K$ with $\dim (\ol{\Y}_{y,i}/O_K) =1$. 
Then, since the morphism 
$$ 
\ol{\Y}_{y,i} - \Y'_{y,i} = \ol{\Y}_{y,i} \cap \D'_i 
\os{\subset}{\lra} \ol{\Y}_{y,i} \cap f^{-1}(H_i) = \ol{\Y}_{y,i} - 
\Y''_{y,i} 
$$ 
is an open and closed immersion, we see that 
$(\Y'_{y,i}, \ol{\Y}_{y,i})$ is also a smooth pair over $O_K$. 
Hence, by assumption, the restriction of $(E,\nabla)$ by 
$(\Y'_{y,i}, \ol{\Y}_{y,i}) \hra (\X'_x, \ol{\X}'_x)$ is overconvergent. 
So the restriction of $(E,\nabla)$ by 
$(\Y''_{y,i}, \ol{\Y}_{y,i}) \hra (\X'_x, \ol{\X}'_x)$ is also 
overconvergent, that is, the restriction of $(E^{\an},\nabla^{\an})$ to 
$\MIC((\wh{\Y}''_{y,i,K}, \wh{\ol{\Y}}_{y,i,K})/K)$ is overconvergent for 
all $y$ and $i$. Hence, by Remark \ref{refined}, we can conclude that 
the restriction of $(E^{\an},\nabla^{\an})$ to 
$\MIC((\wh{\X}''_{x,K}, \wh{\ol{\X}}'_{x,K})/K)$ is overconvergent, 
that is, the pull-back of $(E,\nabla)$ by $j$ 
is overconvergent. Hence we have proved the claim and so the proof of 
the theorem is finished. 
\end{proof}

\end{document}